\documentclass[reqno]{amsart}
\usepackage{amssymb}
\usepackage[usenames, dvipsnames]{color}
\usepackage{pxfonts}
\usepackage[hypertex]{hyperref}
\usepackage{enumerate}

\usepackage{verbatim}

\allowdisplaybreaks[4]

\numberwithin{equation}{section}

\newtheorem{theorem}{Theorem}[section]
\newtheorem{corollary}[theorem]{Corollary}
\newtheorem{lemma}[theorem]{Lemma}
\newtheorem{prop}[theorem]{Proposition}

\theoremstyle{definition}
\newtheorem{remark}[theorem]{Remark}

\theoremstyle{definition}
\newtheorem{definition}[theorem]{Definition}

\theoremstyle{definition}
\newtheorem{assumption}[theorem]{Assumption}

\theoremstyle{definition}

\makeatletter
\def\dashint{\operatorname%
{\,\,\text{\bf-}\kern-.98em\DOTSI\intop\ilimits@\!\!}}
\makeatother

\def\\det{\text{det}}

\def\.5{\frac{1}{2}}

\def\bZ{\mathbb{Z}}

\def\cD{\mathcal{D}}

\newcommand{\RN}[1]{%
  \textup{\uppercase\expandafter{\romannumeral#1}}%
}

\newcommand{\Div}{\operatorname{div}}
\newcommand{\dist}{\text{dist}}

\newcounter{marnote}

\begin{document}


\title[Gradient estimates for divergence form elliptic systems]{Gradient estimates for divergence form elliptic systems arising from composite material}

\author[H. Dong]{Hongjie Dong}
\address[H. Dong]{Division of Applied Mathematics, Brown University, 182 George Street, Providence, RI 02912, USA}
\email{Hongjie\_Dong@brown.edu }
\thanks{H. Dong was partially supported by the NSF under agreement DMS-1600593.}

\author[L. Xu]{Longjuan Xu}
\address[L. Xu]{School of Mathematical Sciences, Beijing Normal University, Laboratory of Mathematics and Complex Systems, Ministry of Education, Beijing 100875, China}
\address{\qquad\quad Division of Applied Mathematics, Brown University, 182 George Street, Providence, RI 02912, USA}
\email{ljxu@mail.bnu.edu.cn, longjuan\_xu@brown.edu}
\thanks{L. Xu was partially supported by the China Scholarship Council (No. 201706040139).}

\begin{abstract}
In this paper, we show that $W^{1,p}$ $(1\leq p<\infty)$ weak solutions to divergence form elliptic systems are Lipschitz and piecewise $C^{1}$ provided that the leading coefficients and data are of piecewise Dini mean oscillation, the lower order coefficients are bounded, and interfacial boundaries are $C^{1,\text{Dini}}$. This extends a result of Li and Nirenberg (\textit{Comm. Pure Appl. Math.} \textbf{56} (2003), 892-925). Moreover, under a stronger assumption on the piecewise $L^{1}$-mean oscillation of the leading coefficients, we derive a global weak type-(1,1) estimate with respect to $A_{1}$ Muckenhoupt weights for the elliptic systems without lower order terms.
\end{abstract}

\maketitle

\section{Introduction and main results}

We consider a composite media with closely spaced interfacial boundaries. The composite media is represented by a bounded domain and divided into a finite number of sub-domains. The physical characteristics of the composite media are smooth in the closure of sub-domains but possibly discontinuous across their boundaries. From the viewpoint of mathematics, these properties are described in terms of a linear second-order divergence type elliptic systems with coefficients which can have jump discontinuities along the boundaries of sub-domains.

To state our main results, we introduce some notation and assumptions. Let $\cD$ be a bounded domain in $\mathbb R^{d}$ that contains $M$ disjoint sub-domains $\cD_{1},\ldots,\cD_{M}$ with $C^{1,\text{Dini}}$ boundaries, that is, $\cD=(\cup_{j=1}^{M}\overline{\cD}_{j})\setminus\partial \cD$. For more details about $C^{1,\text{Dini}}$ boundaries, see Definition \ref{def Dini}. We assume that any point $x\in\cD$ belongs to the boundaries of at most two of the $\cD_{j}$. Hence, if the boundaries of two $\cD_{j}$ touch, then they touch on a whole component of such a boundary. We thus without loss of generality assume that $\partial\cD\subset\partial\cD_{M}$.

Consider the following elliptic systems
\begin{equation}\label{systems}
\mathcal{L}u:=D_{\alpha}(A^{\alpha\beta}D_{\beta}u)+D_{\alpha}(B^{\alpha}u)+\hat{B}^{\alpha}D_{\alpha}u+Cu=\Div g+f,
\end{equation}
where the Einstein summation convention in repeated indices is used,
$$
u=(u^{1},\ldots,u^{n})^{\top},\quad g_{\alpha}=(g_{\alpha}^{1},\ldots,g_{\alpha}^{n})^{\top},\quad f=(f^{1},\ldots,f^{n})^{\top}
$$
are (column) vector-valued functions,
$A^{\alpha\beta}, B^{\alpha}, \hat{B}^{\alpha}$ (often denoted by $A, B, \hat{B}$ for abbreviation), and $C$ are $n\times n$ matrices, which are bounded by a positive constant $\Lambda$, and the leading coefficients matrices $A^{\alpha\beta}$ are uniformly elliptic with ellipticity constant $\nu>0$:
$$\nu|\xi|^{2}\leq A_{ij}^{\alpha\beta}\xi_{\alpha}^{i}\xi_{\beta}^{j},\quad|A^{\alpha\beta}|\leq\nu^{-1}$$
for any $\xi=(\xi_{\alpha}^{i})\in\mathbb R^{n\times d}$. We remark that the hypotheses are also satisfied by the linear systems of elasticity. Recall that a system is called a system of elasticity if $n=d$ and the coefficients satisfy $A_{ij}^{\alpha\beta}=A_{ji}^{\beta\alpha}=A_{\alpha j}^{i\beta}$,
and for all $d\times d$ symmetric matrices $\xi=(\xi_{\alpha}^{i})$,
$$\nu|\xi|^{2}\leq A_{ij}^{\alpha\beta}(x)\xi_{\alpha}^{i}\xi_{\beta}^{j}\leq\nu^{-1}|\xi|^{2},\quad \forall\  x\in \cD.$$

From an engineering point of view, one is interested in deriving bounds on the stress which is represented by $Du$, provided that the principle coefficients $A(x)$ are piecewise constants, given by $A(x)=A_{0}$ for $x$ inside the sub-domains representing the fibers, and $A(x)=1$ elsewhere in $\cD$. We would like to mention that Babu\v{s}ka et al. \cite{basl} studied certain homogeneous isotropic linear systems arising from elasticity. They observed numerically that $|Du|$ is bounded independently of the distance between the regions. When $A_{0}$ is $0$ or $\infty$, it has been shown in many papers that $|Du|$ is unbounded as sub-domains get close; for instance, in \cite{bc,m}. When $A_{0}>0$ is finite, the scalar case when the sub-domains are circular touching fibers of comparable radii was studied by Bonnetier and Vogelius in \cite{bv}. They showed that $|Du|$ remains bounded by using a M\"{o}bius transformation and the maximum principle. A general result concerning the solution to a large class of divergence form elliptic equations with discontinuous coefficients was studied by Li and Vogelius \cite{lv}. There the coefficients $A^{\alpha\beta}$ are assumed to be $C^{\delta}$ ($0<\delta<1$) up to the boundary in each sub-domain with $C^{1,\mu}$ boundary with $\mu\in(0,1]$, but may have jump discontinuities across the boundaries of the sub-domains. The authors derived global Lipschitz and piecewise $C^{1,\delta'}$ estimates of the solution $u$ for $0<\delta'\leq\min\{\delta,\frac{\mu}{d(\mu+1)}\}$. Li and Nirenberg \cite{ln} extended their results to elliptic systems under the same conditions and with $0<\delta'\leq\min\{\delta,\frac{\mu}{2(\mu+1)}\}$.  Thus, a natural question is whether it is possible to further improve the range of $\delta'$. In this paper, we give a definitive answer to the above question.

Denote by $\mathcal{A}$ the set of piecewise constant functions in each $\cD_{j}$, $j=1,\ldots,M$. We then further assume that $A$ is of piecewise Dini mean oscillation in $\cD$, that is,
\begin{align}\label{def omega A}
\omega_{A}(r):=\sup_{x_{0}\in \cD}\inf_{\hat{A}\in\mathcal{A}}\fint_{B_{r}(x_{0})}|A(x)-\hat{A}|\ dx
\end{align}
satisfies the Dini condition, where $B_{r}(x_{0})\subset \cD$. For more details about the Dini condition, see Definition \ref{piece Dini}. For $\varepsilon>0$ small, we set
$$\cD_{\varepsilon}:=\{x\in \cD: \mbox{dist}(x,\partial \cD)>\varepsilon\}.$$

Now, we state the first main result of this paper.
\begin{theorem}\label{thm1}
Let $\cD$ be defined as above. Let $\varepsilon\in (0,1)$, $p\in(1,\infty)$, and $\gamma\in(0,1)$. Assume that $A$, $B$, and $g$ are of piecewise Dini mean oscillation in $\cD$, and $f, g\in L^{\infty}(\cD)$. If $u\in W^{1,p}(\cD)$ is a weak solution to \eqref{systems} in $\cD$, then $u\in C^{1}(\overline{{\cD}_{j}\cap \cD_{\varepsilon}})$, $j=1,\ldots,M$, and $u$ is Lipschitz in $\cD_{\varepsilon}$. Moreover, for any fixed $x\in \cD_{\varepsilon}$, there exists a coordinate system associated with $x$, such that for all $y\in \cD_{\varepsilon}$, we have
\begin{align*}
&|(D_{x'}u(x),U(x))-(D_{x'}u(y),U(y))|\nonumber\\
&\leq N\int_{0}^{|x-y|}\frac{\tilde{\omega}_{g}(t)}{t}\ dt+N|x_{0}-y|^{\gamma}\|(D_{x'}u,U)\|_{L^{1}(\cD)}+ N\int_{0}^{|x-y|}\frac{\tilde{\omega}_{A}(t)}{t}\ dt\nonumber\\
&\quad\cdot\left(\|(D_{x'}u,U)\|_{L^{1}(\cD)}+\int_{0}^{1}\frac{\tilde{\omega}_{g}(t)}{t}\ dt+\|g\|_{L^{\infty}(\cD)}+\|f\|_{L^{\infty}(D)}+\|u\|_{L^{p}(\cD)}\right),
\end{align*}
where $U=A^{d\beta}D_{\beta}u+B^{d}u-g_{d}$, $N$ depends on $n,d,M,p,\Lambda,\nu,\varepsilon,\omega_{B}$, and the $C^{1,\text{Dini}}$ characteristics of $\cD_{j}$, and $\tilde\omega_{\bullet}(t)$ is a Dini function derived from $\omega_{\bullet}(t)$.
\end{theorem}

\begin{corollary}\label{coro}
Let $\cD$ be defined as above and the boundary condition on each sub-domain be replaced by $C^{1,\mu}$. Let $\varepsilon\in (0,1)$ and $p\in(1,\infty)$. Assume that $A, B, g\in C^{\delta}(\overline{\cD}_{j})$ with $\delta\in (0,1)$, and $f\in L^{\infty}(\cD)$. If $u\in W^{1,p}(\cD)$ is a weak solution to \eqref{systems} in $\cD$, then $u\in C^{1}(\overline{{\cD}_{j}\cap \cD_{\varepsilon}})$, $j=1,\ldots,M$. Moreover, for any fixed $x\in \cD_{\varepsilon}$, there exists a coordinate system associated with $x$ such that for all $y\in \cD_{\varepsilon}$, we have
\begin{align}\label{Lip Du}
&|D_{x'}u(x)-D_{x'}u(y)|+|U(x)-U(y)|\nonumber\\
&\leq N|x-y|^{\delta'}\left(\sum_{j=1}^{M}|g|_{\delta;\overline{\cD}_{j}}+\|f\|_{L^{\infty}(\cD)}+\|u\|_{L^{p}(\cD)}\right),
\end{align}
where $0<\delta'=\min\{\delta,\frac{\mu}{\mu+1}\}$, $N$ depends on $n,d,M,\delta,\mu,\nu,\Lambda,\varepsilon,p,\|A\|_{C^{\delta}(\overline{\cD}_{j})},\|B\|_{C^{\delta}(\overline{\cD}_{j})}$, and the $C^{1,\mu}$ norms of $\cD_{j}$.
\end{corollary}

\begin{remark}
In the above results the gradient estimates are independent of the distance between these sub-domains. In \cite{XB13}, Xiong and Bao derived very general BMO, Dini, and H\"older estimates for $H^1$ weak solutions to \eqref{systems}. In particular, for the H\"older estimates, they allowed
$\delta'=\delta=\mu$. However, it appears that the estimates in \cite{XB13} depends on the distance between sub-domains. Since our estimates are independent of the distance, by a similar reasoning as in \cite[Remark 1.2]{ln}, we can obtain \eqref{Lip Du} in the case when more than two of the sub-domains $\overline{\cD}_{j}$ touch, by an approximation argument.
\end{remark}

\begin{remark}
Our result yields $C^{1,\delta'}$ interior estimates for $u$. Actually, we observe from \eqref{Lip Du} that for each $j=1,\ldots,M$, $D_{x'}u, U\in C^{\delta'}(\overline{\cD}_{j}\cap \cD_{\varepsilon})$. Moreover, since
$$D_{d}u=(A^{dd})^{-1}\left(U+g_{d}-B^{d}u-\sum_{\beta=1}^{d-1}A^{d\beta}D_{\beta}u\right),$$
we conclude that $D_{d}u\in C^{\delta'}(\overline{\cD}_{j}\cap \cD_{\varepsilon})$. Compared to \cite{ln}, the range of $\delta'$ is improved.
\end{remark}

\begin{remark}
The conditions on $f$ in Theorem \ref{thm1} and Corollary \ref{coro} can be relaxed. From the proofs below, it is easily seen that we only need $f$ to be in some weaker Morrey space.
\end{remark}

Using the duality argument which is developed in \cite{a,b}, we have the following
\begin{corollary}\label{coro2}
Under the same conditions as in Theorem \ref{thm1}. If $u\in W^{1,1}(\cD)$ is a weak solution to \eqref{systems} in $\cD$, then $u\in W_{\text{loc}}^{1,p}(\cD)$ for some $p\in(1,\infty)$ and the conclusion of Theorem \ref{thm1} still holds.
\end{corollary}

Our arguments and methods are different from those in \cite{ln,lv}. The proofs below are based on Campanato's approach, which was presented in \cite{c,g} and used previously, for instance, in \cite{d,XB13,dk,dek}. The authors in \cite{dk} showed that any weak solution to elliptic equations in divergence form is continuously differentiable if the modulus of continuity of leading coefficients in the $L^{1}$-mean sense satisfies the Dini condition. Recently, Dong, Escauriaza, and Kim \cite{dek} extended and improved the results in \cite{dk} to the boundary for solutions satisfying the zero Dirichlet boundary condition.

The main step of such method is to show that the mean oscillation of $Du$ in balls vanish in a certain order as the radii of balls go to zero. However, we cannot employ this method directly because of the following two obstructions. The first one is the discontinuity of $Du$ in one direction, a situation similar to that in \cite{d}. For a fixed coordinate system, the author in \cite{d} obtained some interior H\"{o}lder regularity of $D_{x'}u$ and $U$ for elliptic systems with coefficients which are H\"{o}lder in $x'$ and measurable in $x^{d}$. Inspired by this, we first choose a coordinate system according to the geometry of the sub-domains, then we consider the elliptic systems with coefficients depending on one variable alone, say $x^{d}$, and derive some interior H\"{o}lder regularity of $D_{x'}u$ and $\bar{U}:=\bar{A}^{d\beta}D_{\beta}u$, where $\bar{A}^{d\beta}$ are piecewise constant matrix-valued functions corresponding to $A$. The second difficulty is that since we only impose the assumptions on the $L^{1}$-mean oscillation of the leading coefficients and data, we cannot use the usual argument based on $L^{p}$ $(p>1)$ estimates. To this end, we make use of a duality argument to derive weak type-$(1,1)$ estimates for solutions to elliptic systems with coefficients depending only on $x^{d}$. Then, we utilize Campanato's method in the $L^{p}$ $(0<p<1)$ setting and some perturbation arguments on $D_{x'}u$ and $U$ together with a certain decomposition of $u$ to get the desired results.

In a forthcoming paper, we will study the second-order elliptic equation in non-divergence form under the same assumptions as that in Theorem \ref{thm1}.

Throughout this paper, unless otherwise stated, $N$ denotes a constant, whose value may vary from line to line and independent of the distance between sub-domains. We call it a {\it{universal constant}}.

The rest of this paper is organized as follows. In Section \ref{preliminaries}, we first fix our domain and the coordinate system associated with a fixed point. For reader's convenience, we introduce some notation, definitions, and lemmas used in this paper. In Section \ref{proof thm1}, we provide the proofs of Theorem \ref{thm1} and Corollary \ref{coro}. In Section \ref{sec_coro2}, we use the duality argument to prove that $u\in W_{\text{loc}}^{1,p}(\cD)$ for some $p\in(1,\infty)$ under the conditions of Corollary \ref{coro2}. Section \ref{sec thm3} is devoted to our second purpose, a global weak type-$(1,1)$ estimates under an additional condition on $\omega_{A}$ and the Dini function introduced in Definition \ref{def Dini}.

\section{Preliminaries}\label{preliminaries}
In this section, we introduce some lemmas which will be used throughout the proofs. Hereafter in this paper, we shall use the following notation and definitions.

\subsection{Notation and definitions}

We write $x=(x^{1},\ldots,x^{d})=(x',x^{d})$, where $d\geq2$. We shall denote
$$B_{r}(x):=\{y\in\mathbb R^{d}: |y-x|<r\},\quad B'_{r}(x'):=\{y'\in\mathbb R^{d-1}: |y'-x'|<r\}.$$
We use $B_{r}:=B_{r}(0)$, $B'_{r}:=B'_{r}(0')$, and $\cD_{r}(x):=\cD\cap B_{r}(x)$ for abbreviation, respectively, where $0\in\mathbb R^{d}$ and $0'\in\mathbb R^{d-1}$. We will also use the following notation:
$$D_{x'}u=u_{x'},\quad DD_{x'}u=u_{xx'}.$$
For a function $f$ defined in $\mathbb R^{d}$, we set
$$(f)_{\cD}=\frac{1}{|\cD|}\int_{\cD}f(x)\ dx=\fint_{\cD}f(x)\ dx,$$
where $|\cD|$ is the $d$-dimensional Lebesgue measure of $\cD$. For $\gamma\in(0,1]$, we denote the $C^{\gamma}$ semi-norm by
$$[u]_{\gamma;\cD}:=\sup_{\begin{subarray}{1}x,y\in\cD\\
x\neq y
\end{subarray}}\frac{|u(x)-u(y)|}{|x-y|^{\gamma}},$$
and the $C^{\gamma}$ norm by
$$|u|_{\gamma;\cD}:=[u]_{\gamma;\cD}+|u|_{0;\cD},\quad \text{where}\,\,|u|_{0;\cD}=\sup_{\cD}|u|.$$

\begin{definition}\label{piece Dini}
We say that a continuous increasing function $\omega: [0,1]\rightarrow\mathbb R$ satisfies the Dini condition provided that $\omega(0)=0$ and
$$\int_{0}^{t}\frac{\omega(s)}{s}\ ds<+\infty,\quad\forall~t\in(0,1).$$
\end{definition}

\begin{definition}\label{def Dini}
Let $\cD\subset\mathbb R^{d}$ be open and bounded. We say that $\partial\cD$ is $C^{1,\text{Dini}}$ if for each point $x_{0}\in\partial\cD$, there exists $R_0\in (0,1/8)$ independent of $x_{0}$ and a $C^{1,\text{Dini}}$ function (that is, $C^{1}$ function whose first derivatives are Dini continuous) $\varphi: B'_{R_0}\rightarrow\mathbb R$ such that (upon relabeling and reorienting the coordinates if necessary) in a new coordinate system $(x',x^{d})$, $x_{0}$ becomes the origin,
$$\cD_{R_{0}}(0)=\{x\in B_{R_{0}}: x^{d}>\varphi(x')\},\quad\varphi(0')=0,$$
and $\nabla_{x'}\varphi$ has a modulus of continuity $\omega_0$, which is increasing, concave, and independent of $x_0$.
\end{definition}

\subsection{Some properties of the domain, coefficients, and data}\label{subsection domain}

Below, we slightly abuse the notation. Consider $\cD$ to be the unit ball $B_{1}$ and take $x_{0}\in B_{3/4}$. By suitable rotation and scaling, we may suppose that a finite number of sub-domains lie in $B_{1}$ and that they take the form
$$x^{d}=h_{j}(x'),\quad\forall~x'\in B'_{1},~j=1,\ldots,l<M,$$
with
\begin{equation}
                            \label{eq10.42}
-1<h_{1}(x')<\dots<h_{l}(x')<1
\end{equation}
and $h_{j}(x')\in C^{1,\text{Dini}}(B'_{1})$. Set $h_{0}(x')=-1$ and $h_{l+1}(x')=1$. Then we have $l+1$ regions:
$$\cD_{j}:=\{x\in \cD: h_{j-1}(x')<x^{d}<h_{j}(x')\},\quad1\leq j\leq l+1.$$ We may suppose that there exists some $\cD_{j_{0}}$, such that $x_{0}\in B_{3/4}\cap \cD_{j_{0}}$ and the closest point on $\partial \cD_{j_{0}}$ to $x_{0}$ is $(x'_{0},h_{j_{0}}(x'_{0}))$, and $\nabla_{x'}h_{j_{0}}(x'_{0})=0'$.
We introduce the $l+1$ ``strips"
$$\Omega_{j}:=\{x\in \cD: h_{j-1}(x'_{0})<x^{d}<h_{j}(x'_{0})\},\quad1\leq j\leq l+1.$$
Then we have the following result.
\begin{lemma}\label{volume}
There exists a constant $N$, depending on $d, l$ and the $C^{1,\text{Dini}}$ characteristics of $h_{j}$, $1\leq j\leq l+1$, such that
$$r^{-d}|(\cD_{j}\Delta\Omega_{j})\cap B_{r}(x_{0})|\leq N\omega_1(r),\quad1\leq j\leq l+1,~0<r<r_0:=\frac 2 3\int_0^{R_0/2}\omega_0'(s)s\,ds,$$
where $\cD_{j}\Delta\Omega_{j}=(\cD_{j}\setminus\Omega_{j})\cup(\Omega_{j}\setminus \cD_{j})$, $\omega_0'$ denotes the left derivative of $\omega_0$, and $\omega_1(r)$ is a Dini function derived from $\omega_0$ in Definition \ref{def Dini}.
\end{lemma}

\begin{proof}
Let $j$ be such that $(x',h_j(x'))\in B_r(x_0)$ for some $x'\in B_1'$. We denote the supremum of $|\nabla_{x'}h_{j}(x')|$ in $B'_{r}(x_0')$ by $\lambda$. Now we fix a point $(y',h_{j}(y'))\in B_r(x_0)$, and without loss of generality, we may assume that $\omega_{0}(r)\leq 1$ for all $0<r<1$. From \eqref{eq10.42} and the fact that $\nabla_{x'}h_{j_{0}}(x_0')=0$,
$$
|h_{j_{0}}(y')-h_{j_{0}}(x_0')|\le r,\quad
|\nabla_{x'}h_{j_{0}}(y')|\le \omega_{0}(2r),\quad
|\nabla_{x'}h_{j}(y')|\ge \lambda-\omega_{0}(2r).
$$
Because of \eqref{eq10.42}, we have
\begin{equation}\label{inetegral00}
\int_{0}^{R}(\lambda-2\omega_0(2r+s))\ ds\leq 3r
\end{equation}
for any $R\in(0,1/8)$. The left-hand side of \eqref{inetegral00} attains its maximum with respect to $R$ when $2\omega_0(2r+R)=\lambda$, which implies that
\begin{equation}\label{equality max}
R\omega_0(2r+R)-\int_{0}^{R}\omega_0(2r+s)\ ds\le 3r/2.
\end{equation}
That is,
\begin{equation}\label{relation r R}
\int_{0}^{R}\omega_0'(2r+s)s\ ds\le 3r/2.
\end{equation}
For a fixed $r$, the left-hand side of \eqref{relation r R} is a increasing function with respect to $R$. To get an upper bound for $\lambda$, we fix the number $R=R(r)(>2r)$ by \eqref{relation r R} such that
\begin{equation}\label{relation r R00}
\int_{0}^{R}\omega_0'(2r+s)s\ ds=3r/2.
\end{equation}
Denote
\begin{equation}\label{def omega1}
\omega_1(r):=\omega_0(2r+R).
\end{equation}
Then combined with $|\nabla_{x'}h_{j}(x')|\le \lambda$ in $B'_{r}(x_0')$, we obtain
\begin{align*}
|(\cD_{j}\Delta\Omega_{j})\cap B_{r}(x_{0})|\leq N\lambda r^d = N\omega_1(r)r^{d}.
\end{align*}

Next we are left to prove that $\omega_1(r)$ is a Dini function on $(0,r_{0})$. For this, it suffices to check the conditions in Definition \ref{piece Dini}. Obviously, \eqref{relation r R00} and \eqref{def omega1} imply that $\omega_1(r)$ is a continuous increasing function on $(0,r_{0})$ and $\omega_1(0)=\omega_0(0)=0$. Moreover, it follows from the increasing property and concavity of $\omega_0$ that
\begin{equation}\label{prop omega}
\int_{0}^{R}\omega_0'(2r+s)s\ ds\geq\frac{R^{2}}{2}\omega_0'(2r+R).
\end{equation}
From \eqref{relation r R00}, we have
\begin{equation}\label{relation dr dR}
\omega_0'(2r+R)R\ dR=\frac{3}{2}\ dr\quad\mbox{for~{\em a.e.}}~r.
\end{equation}
Therefore, in view of \eqref{relation r R00}--\eqref{relation dr dR}, the increasing property and concavity of $\omega_0$ again, and $2r<R$, we have
\begin{align*}
\int_{0}^{r_0}\frac{\omega_1(r)}{r}\ dr&=
\int_{0}^{r_0}\frac{\omega_0(2r+R)}{r}\ dr\\
&=\int_{0}^{R_{0}/2}\frac{R\omega_0(2r+R)\omega_0'(2r+R)}{\int_{0}^{R}\omega_0'(2r+s)s\ ds}\ dR\\
&\leq 2\int_{0}^{R_{0}/2}\frac{\omega_0(2r+R)}{R}\ dR\\
&\leq 2\int_{0}^{R_{0}/2}\frac{\omega_0(2r)+\omega_{0}(R)}{R}\ dR\\
&\leq 4\int_{0}^{R_{0}/2}\frac{\omega_0(R)}{R}\ dR<\infty.
\end{align*}
The lemma is proved.
\end{proof}

In the sequel, we extend $\omega_1(r)$ to be $\omega_1(r_0)$ when $r>r_0$.

\begin{remark}\label{rmk volume}
From the proof above, if we assume that the boundaries of $\cD_{1},\dots,\cD_{M}$ are $C^{1,\mu}$, then $\omega_0(s)=s^{\mu}$ and it follows from \eqref{equality max} that
$$R\leq\left(\frac{3(1+\mu)}{2\mu}r\right)^{\frac{1}{1+\mu}}.$$
Hence,
$$r^{-d}|(\cD_{j}\Delta\Omega_{j})\cap B_{r}(x_{0})|\leq Nr^{\frac{\mu}{1+\mu}},$$
which is the result in \cite[Lemma 5.1]{lv}.
\end{remark}

Let $\hat{A}^{(j)}\in\mathcal{A}$ be a constant function in $\cD_{j}$ which corresponds to the definition of $\omega_{A}(r)$ in \eqref{def omega A}. Similarly, $\hat{B}^{(j)}$ and $\hat{g}^{(j)}$ are defined in $\cD_{j}$. We define the piecewise constant (matrix-valued) functions
\begin{align*}
\bar{A}(x)=
\hat{A}^{(j)},\quad x\in\Omega_{j}.
\end{align*}
Using $\hat{B}^{(j)}$ and $\hat{g}^{(j)}$, we similarly define piecewise constant functions $\bar{B}$ and $\bar{g}$. From Lemma \ref{volume} and the boundedness of $A$, we have
\begin{align}\label{est A}
\fint_{B_{r}(x_{0})}|\hat{A}-\bar{A}|\ dx
\leq Nr^{-d}\sum_{j=1}^{l+1}|(\cD_{j}\Delta\Omega_{j})\cap B_{r}(x_{0})|\leq N\omega_1(r),
\end{align}
which is also true for $\hat{B}$ and $\hat{g}$.

\subsection{Some $L^{p}$ estimates and auxiliary lemmas}

First, let us recall the variably partially small BMO (bounded mean oscillation) condition (see, for instance, \cite{dk2}): there exists a sufficiently small constant $\gamma_{0}=\gamma_{0}(d,n,p,\nu)\in (0,1/2)$ and a constant $r_{0}\in(0,1)$ such that for any $r\in(0,r_{0})$ and $x_{0}\in B_{1}$ with $B_{r}(x_{0})\subset B_{1}$, in a coordinate system depending on $x_{0}$ and $r$, one can find a $\bar{A}(x^{d})$ satisfying
\begin{equation}\label{BMO}
\fint_{B_{r}(x_{0})}|A(x)-\bar A(x^{d})|\ dx\leq\gamma_{0}.
\end{equation}
The following lemma follows from \cite[Theorem 8.6]{dk2} by a standard localization argument which is similar to that in the proof of \cite[Lemma 4]{d}, using the Sobolev embedding theorem and a bootstrap argument.

\begin{lemma}\label{lem loc lq}
Let $p,q\in(1,\infty)$. Assume $A$ satisfies \eqref{BMO} with a sufficiently small constant $\gamma_{0}=\gamma_{0}(d,n,p,q,\nu,\Lambda)\in (0,1/2)$ and $u\in C_{\text{loc}}^{0,1}$ satisfies \eqref{systems} in $B_{1}$, where $f,g\in L^{q}(B_{1})$. Then
$$\|u\|_{W^{1,q}(B_{1/2})}\leq N(\|u\|_{L^{p}(B_{1})}+\|g\|_{L^{q}(B_{1})}+\|f\|_{L^{q}(B_{1})}).$$
In particular, if $q>d$, it holds that
$$|u|_{\gamma;B_{1/2}}\leq N(\|u\|_{L^{p}(B_{1})}+\|g\|_{L^{q}(B_{1})}+\|f\|_{L^{q}(B_{1})}),$$
where $\gamma=1-d/q$ and $N$ depends on $n,d,\nu,\Lambda,p,q$, and $r_{0}$.
\end{lemma}

We next recall the $W^{1,p}$-solvability for elliptic systems with leading coefficients which satisfy \eqref{BMO} in $\bar{B}_{1}$. To be precise, we choose a cut-off function $\eta\in C_{0}^{\infty}(B_{1})$ with
$$0\leq\eta\leq1,\quad \eta\equiv1~\mbox{in}~B_{3/4},\quad|\nabla\eta|\leq 8.$$
Let $\tilde{\mathcal{L}\ }$ be the elliptic operator defined by
$$
\tilde{\mathcal{L}\ }u:=D_{\alpha}(\tilde{A}^{\alpha\beta}D_{\beta}u),
$$
where $\tilde{A}^{\alpha\beta}=\eta A^{\alpha\beta}(x)+\nu(1-\eta)\delta_{\alpha\beta}\delta_{ij}$, $\delta_{\alpha\beta}$ and $\delta_{ij}$ are the Kronecker delta symbols. Consider
\begin{align}\label{approxi sol}
\begin{cases}
\tilde{\mathcal{L}\ }u=\Div g+f&\quad\mbox{in}~B_{1},\\
u=0&\quad\mbox{on}~\partial B_{1},
\end{cases}
\end{align}
where $g, f\in L^{p}(B_{1})$. Then the coefficients $\tilde{A}^{\alpha\beta}(x)$ and the boundary $\partial B_{1}$ satisfy the Assumption 8.1 ($\gamma$) in \cite{dk2} for sufficiently small $\gamma$. Applying \cite[Theorem 8.6]{dk2} to our case, we have
\begin{lemma}\label{solvability}
For any $p\in(1,\infty)$, the following hold.
\begin{enumerate}
\item
For any $u\in W_{0}^{1,p}(B_{1})$,
$$\|u\|_{W^{1,p}(B_{1})}\leq N\big(\|g\|_{L^{p}(B_{1})}+\|f\|_{L^{p}(B_{1})}\big),$$
where $N$ depends on $d,n,p,\nu,\Lambda$, and $r_{0}$.

\item
For any $g, f\in L^{p}(B_{1})$, \eqref{approxi sol} admits a unique solution $u\in  W_{0}^{1,p}(B_{1})$.
\end{enumerate}
\end{lemma}

We note that by \eqref{BMO}, Lemmas \ref{lem loc lq} and \ref{solvability} are applicable in our case. Next we consider systems with coefficients depending on $x^{d}$ alone. We denote
$$\mathcal{L}_{0}u:=D_{\alpha}(\bar{A}^{\alpha\beta}(x^{d})D_{\beta}u),$$
and
$$\bar{U}:=\bar{A}^{d\beta}(x^{d})D_{\beta}u,\quad\mbox{that~is},\quad \bar{U}^{i}=\bar{A}_{ij}^{d\beta}(x^{d})D_{\beta}u^{j},\quad i=1,\ldots,n.$$

\begin{lemma}\label{lemma xn}
Let $p\in(0,\infty)$. Assume $u\in C_{\text{loc}}^{0,1}$ satisfies $\mathcal{L}_{0}u=0$ in $B_{1}$. Then for any $q\in(1,\infty)$, there exists a constant $N=N(n,d,p,q,\nu,\Lambda)$ such that for any matrix-valued constant $\mathbf c\in\mathbb R^{n\times(d-1)}$,
\begin{align}\label{DDu}
\|DD_{x'}u\|_{L^{\infty}(B_{1/2})}\leq N\|D_{x'}u-\mathbf c\|_{L^{p}(B_{1})},
\end{align}
and
\begin{align}\label{DU}
\|D\bar{U}\|_{L^{\infty}(B_{1/2})}\leq N\|D_{x'}u-\mathbf c\|_{L^{p}(B_{1})}.
\end{align}
\end{lemma}

\begin{proof}
It directly follows from Lemma \ref{lem loc lq} that
\begin{align*}
\|u\|_{W^{1,q}(B_{4/5})}\leq N\|u\|_{L^{2}(B_{1})}.
\end{align*}
Then by the Sobolev embedding theorem for $q>d$, we have
\begin{align}\label{infinit u}
\|u\|_{L^{\infty}(B_{4/5})}\leq N\|u\|_{L^{2}(B_{1})}.
\end{align}
For $0<p<1<\infty$, by using the interpolation inequality, we get
$$\|u\|_{L^{2}(B_{5/6})}\leq \|u\|_{L^{p}(B_{5/6})}^{{p}/{2}}\|u\|_{L^{\infty}(B_{5/6})}^{1-{p}/{2}}.$$
Thus, combining \eqref{infinit u} with slightly smaller domain, and H\"{o}lder's inequality, we obtain
\begin{align*}
\|u\|_{L^{\infty}(B_{4/5})}&\leq N\|u\|_{L^{p}(B_{5/6})}^{{p}/{2}}\|u\|_{L^{\infty}(B_{5/6})}^{1-{p}/{2}}\\
&\leq\frac{1}{2}\|u\|_{L^{\infty}(B_{5/6})}+N\|u\|_{L^{p}(B_{5/6})},\quad p>0.
\end{align*}
By a well-known iteration argument (see, for instance, \cite[Lemma 3.1 of Ch.V]{g}), we get
\begin{align}\label{u infinite}
\|u\|_{L^{\infty}(B_{4/5})}\leq N\|u\|_{L^{p}(B_{1})},\quad p>0.
\end{align}

Now, we define the finite difference quotient
$$
\delta_{h,k}f(x):=\frac{f(x+he_{k})-f(x)}{h}
$$
where $k=1,\dots,d-1$, $0<|h|<1/12$. Since $\bar{A}^{\alpha\beta}(x^{d})$ are independent of $x'$, we have $\mathcal{L}_{0}(\delta_{h,k}u)=0$ in $B_{1}$. We thus use Lemma \ref{lem loc lq} and \eqref{u infinite} to get
\begin{align*}
\|\delta_{h,k}u\|_{W^{1,q}(B_{1/2})}\leq N\|\delta_{h,k}u\|_{L^{2}(B_{2/3})}\leq N\|D_{x'}u\|_{L^{2}(B_{3/4})}
\leq N\|u\|_{L^{2}(B_{4/5})}\leq N\|u\|_{L^{p}(B_{1})},\quad p>0.
\end{align*}
Letting $h\rightarrow0$, we obtain
\begin{align}\label{Dx'u WW}
\|D_{x'}u\|_{W^{1,q}(B_{1/2})}\leq N\|u\|_{L^{p}(B_{1})},\quad p>0.
\end{align}
Moreover, notice that in $B_{1}$,
\begin{align}\label{Dx U}
D_{d}\bar{U}=-\sum_{\alpha=1}^{d-1}\sum_{\beta=1}^{d}\bar{A}^{\alpha\beta}D_{\alpha\beta}u,\quad D_{x'}\bar{U}=\sum_{\beta=1}^{d}\bar{A}^{d\beta}D_{x'}D_{\beta}u.
\end{align}
Then by using Lemma \ref{lem loc lq}, \eqref{u infinite}, \eqref{Dx'u WW}, and the boundedness of $\bar{A}$, we obtain
\begin{align}\label{esti U}
\|\bar{U}\|_{W^{1,q}(B_{1/2})}&=\|\bar{U}\|_{L^{q}(B_{1/2})}+\|D\bar{U}\|_{L^{q}(B_{1/2})}\nonumber\\
&\leq N\|Du\|_{L^{q}(B_{1/2})}+N\Big(\|D_{x'}\bar{U}\|_{L^{q}(B_{1/2})}+\|D_{d}\bar{U}\|_{L^{q}(B_{1/2})}\Big)\nonumber\\
&\leq N\|u\|_{W^{1,q}(B_{1/2})}+N\|DD_{x'}u\|_{L^{q}(B_{1/2})}\nonumber\\
&\leq N\|u\|_{L^{2}(B_{2/3})}+N\|D_{x'}u\|_{W^{1,q}(B_{1/2})}\leq N\|u\|_{L^{p}(B_{1})}.
\end{align}
Combining \eqref{Dx'u WW} and \eqref{esti U}, we get
$$\|D_{x'}u\|_{W^{1,q}(B_{1/2})}+\|\bar{U}\|_{W^{1,q}(B_{1/2})}\leq N\|u\|_{L^{p}(B_{1})},\quad p>0.$$
By the Sobolev embedding theorem for $q>d$, we have
$$\|D_{x'}u\|_{L^{\infty}(B_{1/2})}+\|\bar{U}\|_{L^{\infty}(B_{1/2})}\leq N\|u\|_{L^{p}(B_{1})},\quad p>0.$$
By using $\bar{A}^{dd}(x^{d})\geq\nu$, we then have
$$\|Du\|_{L^{\infty}(B_{1/2})}\leq N\|u\|_{L^{p}(B_{1})},\quad p>0.$$
Similarly, by using the fact that the coefficients of $\mathcal{L}_{0}$ are independent of $x'$ and hence, for any matrix-valued constant $\mathbf c\in\mathbb R^{n\times(d-1)}$, we have $\mathcal{L}_{0}(D_{x'}u-\mathbf c)=0$ in $B_{1}$. Then
$$\|DD_{x'}u\|_{L^{\infty}(B_{1/2})}\leq N\|D_{x'}u-\mathbf c\|_{L^{p}(B_{1})},\quad p>0.$$
Thus, \eqref{DDu} is proved and \eqref{DU} follows from \eqref{Dx U} and \eqref{DDu}.
\end{proof}

Finally, we will also use the following lemmas, which are Lemma 2.7 in \cite{dk} and Lemma 4.1 in \cite{dek}.

\begin{lemma}\label{lemma omiga}
Let $\omega$ be a nonnegative bounded function. Suppose there is $c_{1},c_{2}>0$ and $0<\kappa<1$ such that for $\kappa t\leq s\leq t$ and $0<t<r$,
\begin{align}\label{equivalence}
c_{1}\omega(t)\leq \omega(s)\leq c_{2}\omega(t).
\end{align}
Then, we have
$$\sum_{i=0}^{\infty}\omega(\kappa^{i}r)\leq N\int_{0}^{r}\frac{\omega(t)}{t}\ dt,$$
where $N=N(\kappa,c_{1},c_{2})$.
\end{lemma}

\begin{lemma}\label{lemma weak}
Let $\cD$ be a bounded domain in $\mathbb R^{d}$ satisfying
\begin{equation}\label{condition D}
|\cD_{r}(x)|\geq A_{0}r^{d}\quad\mbox{for~all}~x\in\overline{\cD}~\mbox{and}~r\in(0,\mbox{diam}\ \cD],
\end{equation}
where $A_{0}>0$ is a constant. Let $p\in (1,\infty)$ and $T$ be a bounded linear operator on $L^{p}(\cD)$. Suppose that for any $\bar{y}\in \cD$ and $0<r<\mu~ \mbox{diam}\ \cD$,  we have
$$\int_{\cD\setminus B_{cr}(\bar{y})}|Tb|\leq C_{0}\int_{\cD_{r}(\bar{y})}|b|$$
whenever $b\in L^{p}(\cD)$ is supported in $\cD_{r}(\bar{y})$, $\int_{\cD}b=0$, and $c>1$, $C_{0}>0$, $\mu\in(0,1)$ are constants. Then for $g\in L^{p}(\cD)$ and any $t>0$, we have
$$|\{x\in \cD: |Tg(x)|>t\}|\leq\frac{N}{t}\int_{\cD}|g|,$$
where $N=N(d,c,C_{0},\cD,A_{0},\mu,\|T\|_{L^p\rightarrow L^p})$ is a constant.
\end{lemma}

We note that in \cite[Lemma 4.1]{dek} the exponent $p$ is assumed to be $2$. However, by a slight modification of the proof there, the result can be easily extended to general $p\in (1,\infty)$. See also Lemma \ref{weak general} below by taking $w=1$.

\section{Proofs of Theorem \ref{thm1} and Corollary \ref{coro}}\label{proof thm1}

In this section, we give the proofs of Theorem \ref{thm1} and Corollary \ref{coro}.

\subsection{Proof of Theorem \ref{thm1}}

We first show that by using the global $W^{2,p}$ estimate for the Laplacian operator, Sobolev embedding theorem, and Lemma \ref{lem loc lq}, we only need to consider systems without lower order terms. We rewrite \eqref{systems} as
$$D_{\alpha}(A^{\alpha\beta}D_{\beta}u)=\Div (g-Bu+Dv)\quad \mbox{in}~\cD=B_{1}\subset B_{2},$$
where $v\in W^{2,p}(B_{2})$ is a strong solution to
\begin{align*}
\begin{cases}
\Delta v=(f-\hat{B}^{\alpha}D_{\alpha}u-Cu)\chi_{B_{1}}&\quad\mbox{in}~B_{2},\\
v=0&\quad\mbox{on}~\partial B_{2}.
\end{cases}
\end{align*}
Then by the global $W^{2,p}$ estimate, we have
\begin{align}\label{W2p v}
\|v\|_{W^{2,p}(B_{2})}&\leq N\|f-\hat{B}^{\alpha}D_{\alpha}u-Cu\|_{L^{p}(B_{1})}\nonumber\\
&\leq N\left(\|u\|_{W^{1,p}(B_{1})}+\|f\|_{L^{\infty}(B_{1})}\right).
\end{align}
By Lemma \ref{lem loc lq}, the Sobolev embedding theorem, and \eqref{W2p v}, we have
\begin{align*}
\|u\|_{W^{1,p^{*}}(B_{1/2})}&\leq N\left(\|u\|_{L^{p}(B_{1})}+\|g-Bu+Dv\|_{L^{p^{*}}(B_{1})}\right)\\
&\leq N\left(\|u\|_{W^{1,p}(B_{1})}+\|g\|_{L^{\infty}(B_{1})}+\|Dv\|_{W^{1,p}(B_{1})}\right)\\
&\leq N\left(\|u\|_{W^{1,p}(B_{1})}+\|g\|_{L^{\infty}(B_{1})}+\|f\|_{L^{\infty}(B_{1})}\right),
\end{align*}
where ${1}/{p^{*}}={1}/{p}-{1}/{d}$ for $p<d$ and $p^{*}\in(p,\infty)$ is arbitrary if $p\geq d$. By a bootstrap argument, we can find some $q>d$ so that $u\in W_{\text{loc}}^{1,q}(B_{1})$, and
$$\|u\|_{W^{1,q}(B_{1/2})}\leq N\left(\|u\|_{W^{1,p}(B_{1})}+\|g\|_{L^{\infty}(B_{1})}+\|f\|_{L^{\infty}(B_{1})}\right).$$
By Morrey's inequality, we obtain $u\in C^{\gamma}(B_{1/2})$, where $\gamma=1-{d}/{q}$. Noticing that $f-\hat{B}^{\alpha}D_{\alpha}u-Cu\in L_{\text{loc}}^{q}(B_{1})$ for $q>d$, then by \eqref{W2p v} with $q$ in place of $p$, and Morrey's inequality, we get $Dv\in C^{\gamma}(B_{1/2})$ and
$$\|Dv\|_{C^{\gamma}(B_{1/2})}\leq N\left(\|u\|_{W^{1,p}(B_{1})}+\|g\|_{L^{\infty}(B_{1})}+\|f\|_{L^{\infty}(B_{1})}\right).$$
Therefore, we conclude that $u\in W^{1,p}(B_{1})$ is a weak solution of
$$D_{\alpha}(A^{\alpha\beta}D_{\beta}u)=\Div g'\quad \mbox{in}~B_{1}$$
without lower order terms, where $g':=g-Bu+Dv$ with
$$\|g'\|_{L^{\infty}(B_{1/2})}\leq N\left(\|u\|_{W^{1,p}(B_{1})}+\|g\|_{L^{\infty}(B_{1})}+\|f\|_{L^{\infty}(B_{1})}\right).$$
Moreover, $g'$ is also of piecewise Dini mean oscillation with
\begin{align*}
\omega_{g'}(r)&\leq N(\Lambda)\big(\omega_{g}(r)+\omega_{B}(r)\sup_{B_{1/2}}u+r^{\gamma}
[u]_{\gamma;B_{1/2}}+r^{\gamma}[Dv]_{\gamma;B_{1/2}}\big)\\
&\leq N\Big(\omega_{g}(r)+\omega_{B}(r)\sup_{B_{1/2}}u+r^{\gamma}
\big(\|u\|_{W^{1,p}(B_{1})}+\|g\|_{L^{\infty}(B_{1})}
+\|f\|_{L^{\infty}(B_{1})}\big)\Big).
\end{align*}

Hence, in order to prove Theorem \ref{thm1}, it suffices to prove the following

\begin{prop}\label{main prop}
Let $\varepsilon\in(0,1)$ and $p\in(1,\infty)$. Suppose that $A$ and $g$ are of piecewise Dini mean oscillation in $\cD$, $g\in L^{\infty}(\cD)$. If $u\in W^{1,p}(\cD)$ is a weak solution to
$$D_{\alpha}(A^{\alpha\beta}\cD_{\beta}u)=\Div g\quad\mbox{in}~\cD,$$
then $u\in C^{1}(\overline{{\cD}_{j}\cap \cD_{\varepsilon}})$, $j=1,\ldots,M$, and $u$ is Lipschitz in $\cD_{\varepsilon}$.
\end{prop}

The rest of this subsection is devoted to proving Proposition \ref{main prop} by adapting Campanato's approach \cite{c,g}. We shall derive a priori estimate of the modulus of continuity of $(D_{x'}u,U)$ by assuming that $u\in C^{0,1}(B_{3/4})$. The general case follows from an approximation argument and will be outlined at the end of the proof. We fix $x_{0}\in B_{3/4}\cap \cD_{j_{0}}$, $0<r\leq 1/4$, and take a coordinate system associated with $x_{0}$ as in Subsection \ref{subsection domain}. Denote
\begin{align*}
\bar{\ \mathcal{L}_{x'_{0}}}u:=D_{\alpha}(\bar{A}^{\alpha\beta}(x'_{0},x^{d})D_{\beta}u).
\end{align*}
We present a series of lemmas and their proofs which will provide key estimates for the proof of Proposition \ref{main prop}. We modify the coefficients $\bar{A}^{\alpha\beta}(x'_{0},x^{d})$ to get the following elliptic operator defined by
$$
\tilde{\mathcal{L}\ }u:=D_{\alpha}(\tilde{A}^{\alpha\beta}D_{\beta}u),
$$
where $\tilde{A}^{\alpha\beta}=\eta \bar{A}^{\alpha\beta}(x'_{0},x^{d})+\nu(1-\eta)\delta_{\alpha\beta}\delta_{ij}$ with $\eta\in C_{0}^{\infty}(B_{r}(x_{0}))$ satisfying
$$0\leq\eta\leq1,\quad\eta\equiv1~\mbox{in}~B_{2r/3}(x_{0}),
\quad|\nabla\eta|\leq {6}/{r}.$$
Then the $W^{1,p}$-solvability and estimates for the operator $\tilde{\mathcal{L}\ }$ follow from Lemma \ref{solvability} with a scaling.

\begin{lemma}\label{weak est barv}
Let $p\in(1,\infty)$. Let $v\in W^{1,p}(B_{r}(x_{0}))$ be a weak solution to the problem
\begin{align*}
\begin{cases}
\tilde{\mathcal{L}\ }v=\Div(F\chi_{B_{r/2}(x_{0})})&\quad\mbox{in}~B_{r}(x_{0}),\\
v=0&\quad\mbox{on}~\partial B_{r}(x_{0}),
\end{cases}
\end{align*}
where $F\in L^{p}(B_{r/2}(x_{0}))$. Then for any $t>0$, we have
\begin{align*}
|\{x\in B_{r/2}(x_{0}): |Dv(x)|>t\}|\leq\frac{N}{t}\|F\|_{L^{1}(B_{r/2}(x_{0}))},
\end{align*}
where $N>0$ is a {\it{universal constant}}.
\end{lemma}

\begin{proof}
For simplicity, we set $x_{0}=0$, $r=1$, $\bar{A}^{\alpha\beta}(x^{d}):=\bar{A}^{\alpha\beta}(0',x^{d})$, and $\bar{\mathcal{L}\ }:=\bar{\ \mathcal{L}_{0'}}$. Suppose $E=(E^{\alpha\beta}(x^{d}))$ is a $d\times d$ matrix with
\begin{align*}
E^{\alpha\beta}(x^{d})&=\delta_{\alpha\beta}~ \mbox{for}~\alpha, \beta\in\{1,\ldots,d-1\};\quad E^{\alpha d}(x^{d})=\bar{A}^{d\alpha}(x^{d})~ \mbox{for}~\alpha\in\{1,\ldots,d\};\\
E^{d\beta}(x^{d})&=0~\mbox{for}~\beta\in\{1,\ldots,d-1\}.
\end{align*}
For any $\hat F\in L^{p}(B_{1/2})$, let $F=E\hat F$ and solve for $v$. By Lemma \ref{solvability}, we can see that $T:\hat F\to Dv$ is a bounded linear operator on $L^{p}(B_{1/2})$. It suffices to show that $T$ satisfies the hypothesis of Lemma \ref{lemma weak}. We set $c=24$ and fix $\bar{y}\in B_{1/2}$, $0<r<1/4$. Let $\hat{b}\in L^{p}(B_{1})$ be supported in $B_{r}(\bar{y})\cap B_{1/2}$ with mean zero, $b=E\hat{b}$, and $v_{1}\in W^{1,p}(B_{1})$ be the unique weak solution of
\begin{align*}
\begin{cases}
\tilde{\mathcal{L}\ }v_{1}=\Div b&\quad\mbox{in}~B_{1},\\
v_{1}=0&\quad\mbox{on}~\partial B_{1}.
\end{cases}
\end{align*}
For any $R\geq cr$ such that $B_{1/2}\setminus B_{R}(\bar{y})\neq\emptyset$ and $h\in C_{0}^{\infty}((B_{2R}(\bar{y})\setminus B_{R}(\bar{y}))\cap B_{1/2})$, let $v_{0}\in W^{1,p'}(B_{1})$ be a weak solution of
\begin{align*}
\begin{cases}
\tilde{\mathcal{L}^{*}\ }v_{0}=\Div h&\quad\mbox{in}~B_{1},\\
v_{0}=0&\quad\mbox{on}~\partial B_{1},
\end{cases}
\end{align*}
where $\tilde{\mathcal{L}^{*}\ }$ is the adjoint operator of $\tilde{\mathcal{L}\ }$, ${1}/{p}+{1}/{p'}=1$. It follows from the definition of weak solutions and the assumption of $\hat b$ that
\begin{align*}
\int_{B_{1/2}}Dv_{1}\cdot h&=\int_{B_{1/2}}Dv_{0}\cdot b\\
&=\int_{B_{r}(\bar{y})\cap B_{1/2}}
\left(
          D_{x'}v_{0}, V_{0}
 \right)\cdot\hat{b}\\
&=\int_{B_{r}(\bar{y})\cap B_{1/2}}
\left(
          D_{x'}v_{0}-D_{x'}v_{0}(\bar{y}), V_{0}-V_{0}(\bar{y})
 \right)\cdot\hat{b},
\end{align*}
where
$V_{0}=\bar{A}^{d\beta}(x^{d})D_{\beta}v_{0}$. Therefore, we have
\begin{align}\label{esti Dv h}
&\left|\int_{(B_{2R}(\bar{y})\setminus B_{R}(\bar{y}))\cap B_{1/2}}Dv_{1}\cdot h\right|\nonumber\\
&\leq \|\hat{b}\|_{L^{1}(B_{r}(\bar{y})\cap B_{1/2})}\left|\left|\left(
          D_{x'}v_{0}-D_{x'}v_{0}(\bar{y}), V_{0}-V_{0}(\bar{y})
 \right)\right|\right|_{L^{\infty}(B_{r}(\bar{y})\cap B_{1/2})}.
\end{align}
Recalling that $\eta\equiv1$ in $B_{2/3}$ and $B_{R/12}(\bar{y})\subset B_{2/3}$, we conclude that $v_{0}\in W^{1,p'}(B_{1})$ satisfies
$$\bar{\mathcal{L}^{*}\ }v_{0}=0\quad\mbox{in}~B_{R/12}(\bar{y}).$$
Then, by using \eqref{DDu}, \eqref{DU} with a suitable scaling, $r\leq R/24$, and the $L^{p}$ estimate, we have
\begin{align*}
&\|D_{x'}v_{0}-D_{x'}v_{0}(\bar{y})\|_{L^{\infty}(B_{r}(\bar{y})\cap B_{1/2})}
+\|V_{0}-V_{0}(\bar{y})\|_{L^{\infty}(B_{r}(\bar{y})\cap B_{1/2})}\\
&\leq NrR^{-1-{d}/{p'}}\|Dv_{0}\|_{L^{p'}(B_{R/12}(\bar{y}))}\\
&\leq NrR^{-1-{d}/{p'}}\|h\|_{L^{p'}((B_{2R}(\bar{y})\setminus B_{R}(\bar{y}))\cap B_{1/2})}.
\end{align*}
Now, coming back to \eqref{esti Dv h} and using the duality, we have
$$\|Dv_{1}\|_{L^{p}((B_{2R}(\bar{y})\setminus B_{R}(\bar{y}))\cap B_{1/2})}\leq NrR^{-1-{d}/{p'}}\|\hat{b}\|_{L^{1}(B_{r}(\bar{y})\cap B_{1/2})}.$$
Thus, it follows from H\"{o}lder's inequality that
\begin{align}\label{dilation Dv}
\|Dv_{1}\|_{L^{1}((B_{2R}(\bar{y})\setminus B_{R}(\bar{y}))\cap B_{1/2})}\leq NrR^{-1}\|\hat{b}\|_{L^{1}(B_{r}(\bar{y})\cap B_{1/2})}.
\end{align}
Let $N_{0}$ be the smallest positive integer such that $B_{1/2}\subset B_{2^{N_{0}}cr}(\bar{y})$. By taking $R=cr, 2cr,\ldots,2^{N_{0}-1}cr$ in \eqref{dilation Dv} and summarizing, we have
\begin{align*}
\int_{B_{1/2}\setminus B_{cr}(\bar{y})}|Dv_{1}|\ dx&\leq N\sum_{k=1}^{N_{0}}2^{-k}\|\hat{b}\|_{L^{1}(B_{r}(\bar{y})\cap B_{1/2})}\leq N\int_{B_{r}(\bar{y})\cap B_{1/2}}|\hat{b}|\ dx.
\end{align*}
Therefore, $T$ satisfies the hypothesis of Lemma \ref{lemma weak}. The lemma is proved.
\end{proof}

Consider
$$\phi(x_{0},r):=\inf_{\mathbf q\in\mathbb R^{n\times d}}\left(\fint_{B_{r}(x_{0})}|(D_{x'}u,U)-\mathbf q|^{q}\ dx\right)^{1/q},$$
where $0<q<1$ is some fixed exponent. First of all, by using H\"{o}lder's inequality, we have
\begin{align}\label{est phi}
\phi(x_{0},r)\leq\left(\fint_{B_{r}(x_{0})}|(D_{x'}u,U)|^{q}\ dx\right)^{1/q}\leq Nr^{-d}\|(D_{x'}u,U)\|_{L^{1}(B_{r}(x_{0}))},
\end{align}
where $N=N(d)$.

\begin{lemma}\label{lemma itera}
For any $\gamma\in (0,1)$ and $0<\rho\leq r\leq 1/4$, we have
\begin{align}\label{est phi'}
\phi(x_{0},\rho)\leq N\Big(\frac{\rho}{r}\Big)^{\gamma}r^{-d}\|(D_{x'}u,U)\|_{L^{1}(B_{r}(x_{0}))}+N\tilde{\omega}_{A}(\rho)\|Du\|_{L^{\infty}(B_{r}(x_{0}))}+N\tilde{\omega}_{g}(\rho),
\end{align}
where $N>0$ is a \it{universal constant}, and $\tilde\omega_{\bullet}(t)$ is a Dini function derived from $\omega_{\bullet}(t)$.
\end{lemma}

\begin{proof}
For any $t>0$, by using Lemma \ref{weak est barv} with $F=(\bar{A}(x'_{0},x^{d})-A(x))Du+g(x)-\bar{g}(x'_{0},x^{d})$, and \eqref{est A}, we have
\begin{align}\label{weak type Dv bar}
&|\{x\in B_{r/2}(x_{0}): |Dv(x)|>t\}|\nonumber\\
&\leq\frac{N}{t}\int_{B_{r}(x_{0})}|F|\ dx\nonumber\\
&\leq\frac{N}{t}\left(\int_{B_{r}(x_{0})}|g(x)-\bar{g}(x'_{0},x^{d})|\ dx+\int_{B_{r}(x_{0})}|(A(x)-\bar{A}(x'_{0},x^{d}))Du|\ dx\right)\nonumber\\
&\leq\frac{N}{t}\Bigg(\int_{B_{r}(x_{0})}|g(x)-\hat{g}|\ dx+\int_{B_{r}(x_{0})}|\hat{g}-\bar{g}(x'_{0},x^{d})|\ dx\nonumber\\
&\quad+\Big(\int_{B_{r}(x_{0})}|A(x)-\hat{A}|\ dx+\int_{B_{r}(x_{0})}|\hat{A}-\bar{A}(x'_{0},x^{d})|\ dx\Big)\|Du\|_{L^{\infty}(B_{r}(x_{0}))}\Bigg)\nonumber\\
&\leq\frac{Nr^{d}}{t}\Big(\omega_{g}(r)+\omega_1(r)
+\Big(\omega_{A}(r)+\omega_1(r)\Big)
\|Du\|_{L^{\infty}(B_{r}(x_{0}))}\Big)\nonumber\\
&\leq\frac{Nr^{d}}{t}\Big(\bar\omega_{g}(r)
+\bar\omega_{A}(r)\|Du\|_{L^{\infty}(B_{r}(x_{0}))}\Big),
\end{align}
where $\bar\omega_{\bullet}(r)=\omega_{\bullet}(r)+\omega_1(r)$. For any given $0<q<1$, it follows from the boundedness of $\bar{A}$,
\begin{align*}
\int_{B_{r/2}(x_{0})}|Dv|^{q}\ dx=\int_{0}^{\infty}qt^{q-1}|\{x\in B_{r/2}(x_{0}): |Dv(x)|>t\}|\ dt
\end{align*}
and \eqref{weak type Dv bar} that
\begin{align*}
&\int_{B_{r/2}(x_{0})}|D_{x'}v|^{q}\ dx+\int_{B_{r/2}(x_{0})}|V|^{q}\ dx\\
&\leq N\int_{0}^{\infty}qt^{q-1}|\{x\in B_{r/2}(x_{0}): |Dv(x)|>t\}|\ dt\\
&\leq N\left(\int_{0}^{\tau}+\int_{\tau}^{\infty}\right)qt^{q-1}|\{x\in B_{r/2}(x_{0}): |Dv(x)|>t\}|\ dt\\
&\leq N\tau^{q}|B_{r}(x_{0})|+\frac{Nq}{1-q}\tau^{q-1}\Big(r^{d}\bar\omega_{g}(r)+r^{d}\bar\omega_{A}(r)\|Du\|_{L^{\infty}(B_{r}(x_{0}))}\Big),
\end{align*}
where $V=\bar{A}^{d\beta}(x'_{0},x^{d})D_{\beta}v(x)$ and $\tau=\frac{r^{d}}{|B_{r}(x_{0})|}\Big(\bar\omega_{g}(r)+\bar\omega_{A}(r)\|Du\|_{L^{\infty}(B_{r}(x_{0}))}\Big)$. Then we have
\begin{align*}
&\int_{B_{r/2}(x_{0})}|D_{x'}v|^{q}\ dx+\int_{B_{r/2}(x_{0})}|V|^{q}\ dx\\
&\leq\frac{N}{1-q}|B_{r}(x_{0})|^{1-q}\Big(r^{d}\bar\omega_{g}(r)+r^{d}\bar\omega_{A}(r)\|Du\|_{L^{\infty}(B_{r}(x_{0}))}\Big)^{q},
\end{align*}
which implies
\begin{align}\label{holder v bar}
\left(\fint_{B_{r/2}(x_{0})}|D_{x'}v|^{q}\ dx+\fint_{B_{r/2}(x_{0})}|V|^{q}\ dx\right)^{{1}/{q}}\leq N\Big(\bar\omega_{A}(r)\|Du\|_{L^{\infty}(B_{r}(x_{0}))}+\bar\omega_{g}(r)\Big).
\end{align}

Let
$$
u_{1}(x^{d})=\int_{-1}^{x^{d}}(\bar{A}^{dd}(x'_{0},s))^{-1}\bar{g}_{d}(x'_{0},s)\,ds,
\quad \bar{u}=u-u_{1},\quad w=\bar{u}-v,
$$
so that $w$ satisfies $\bar{\ \mathcal{L}_{x'_{0}}}w=0$ in $B_{r/2}(x_{0})$. Noticing that for any $\mathbf q=(\mathbf q',q_{d})\in \mathbb R^{n\times d}$, the same system is satisfied by $D_{\beta}w-q_{\beta}$ for $\beta=1,\ldots,d-1$. By Lemma \ref{lemma xn} with a suitable scaling, we have
\begin{align}\label{DW}
&\|DD_{\beta}w\|_{L^{\infty}(B_{r/4}(x_{0}))}^{q}
+\|DW\|_{L^{\infty}(B_{r/4}(x_{0}))}^{q}\nonumber\\
&\leq Nr^{-(d+q)}\int_{B_{r/2}(x_{0})}|D_{x'}w-\mathbf q'|^{q}\ dx\nonumber\\
&\leq Nr^{-(d+q)}\int_{B_{r/2}(x_{0})}|(D_{x'}w,W)-\mathbf q|^{q}\ dx,
\end{align}
where $W=\bar{A}^{d\beta}(x'_{0},x^{d})D_{\beta}w$. Thus, for any $\kappa\in\big(0,1/4\big)$, by \eqref{DW}, we have
\begin{align*}
&\|D_{x'}w-(D_{x'}w)_{B_{\kappa r}(x_{0})}\|_{L^{q}(B_{\kappa r}(x_{0}))}^{q}+\|W-(W)_{B_{\kappa r}(x_{0})}\|_{L^{q}(B_{\kappa r}(x_{0}))}^{q}\nonumber\\
&\leq N(\kappa r)^{d+q}\left(\|DD_{x'}w\|_{L^{\infty}(B_{r/4}(x_{0}))}^{q}
+\|DW\|_{L^{\infty}(B_{r/4}(x_{0}))}^{q}\right)\nonumber\\
&\leq N\kappa^{d+q}\int_{B_{r/2}(x_{0})}|(D_{x'}w,W)-\mathbf q|^{q}\ dx,
\end{align*}
which implies
\begin{align}\label{holder w bar}
&\left(\fint_{B_{\kappa r}(x_{0})}|D_{x'}w-(D_{x'}w)_{B_{\kappa r}(x_{0})}|^{q}\ dx+\fint_{B_{\kappa r}(x_{0})}|W-(W)_{B_{\kappa r}(x_{0})}|^{q}\ dx\right)^{{1}/{q}}\nonumber\\
&\leq N_{0}\kappa\left(\fint_{B_{r/2}(x_{0})}|(D_{x'}w,W)-\mathbf{q}|^{q}\ dx\right)^{{1}/{q}},
\end{align}
where $N_{0}>0$ is a {\it{universal constant}}. Recalling that $\bar{u}=w+v$, we obtain from \eqref{holder w bar} that
\begin{align}\label{iteration u bar}
&\left(\fint_{B_{\kappa r}(x_{0})}|(D_{x'}\bar{u},\bar{U})-((D_{x'}w)_{B_{\kappa r}(x_{0})},(W)_{B_{\kappa r}(x_{0})})|^{q}\ dx\right)^{{1}/{q}}\nonumber\\
&\leq\left(\fint_{B_{\kappa r}(x_{0})}|D_{x'}\bar{u}-(D_{x'}w)_{B_{\kappa r}(x_{0})}|^{q}+|\bar{U}-(W)_{B_{\kappa r}(x_{0})}|^{q}\ dx\right)^{{1}/{q}}\nonumber\\
&\leq2^{{1}/{q}-1}\left(\fint_{B_{\kappa r}(x_{0})}|D_{x'}w-(D_{x'}w)_{B_{\kappa r}(x_{0})}|^{q}+|W-(W)_{B_{\kappa r}(x_{0})}|^{q}\ dx\right)^{{1}/{q}}\nonumber\\
&\quad+N\left(\fint_{B_{\kappa r}(x_{0})}|D_{x'}v|^{q}+|V|^{q}\ dx\right)^{{1}/{q}}\nonumber\\
&\leq N_{0}\kappa\left(\fint_{B_{r/2}(x_{0})}|(D_{x'}\bar{u},\bar{U})-\mathbf{q}|^{q}\ dx\right)^{{1}/{q}}+N\kappa^{-{d}/{q}}
\left(\fint_{B_{r/2}(x_{0})}|D_{x'}v|^{q}+|V|^{q}\ dx\right)^{{1}/{q}},
\end{align}
where $\bar{U}=\bar{A}^{d\beta}(x'_{0},x^{d})D_{\beta}\bar{u}$. Recalling that $D_{x'}\bar{u}=D_{x'}u$, $U=A^{d\beta}(x)D_{\beta}u-g_{d}(x)$ and $\bar{U}=\bar{A}^{d\beta}(x'_{0},x^{d})D_{\beta}u-\bar{g}_{d}(x'_{0},x^{d})$, we have for $x\in B_{r}(x_{0})$,
\begin{align*}
|U-\bar{U}|\leq \|Du\|_{L^{\infty}(B_{r}(x_{0}))}|A(x)-\bar{A}(x'_{0},x^{d})|+|g_{d}(x)-\bar{g}_{d}(x'_{0},x^{d})|.
\end{align*}
Thus, coming back to \eqref{iteration u bar}, using \eqref{est A} and \eqref{holder v bar}, we have
\begin{align*}
&\left(\fint_{B_{\kappa r}(x_{0})}|(D_{x'}u,U)-((D_{x'}w)_{B_{\kappa r}(x_{0})},(W)_{B_{\kappa r}(x_{0})})|^{q}\ dx\right)^{1/q}\\
&\leq N_{0}\kappa\left(\fint_{B_{r}(x_{0})}|(D_{x'}u,U)-\mathbf{q}|^{q}\ dx\right)^{1/q}+N\kappa^{-d/q}
\left(\fint_{B_{r}(x_{0})}|U-\bar{U}|^{q}\ dx\right)^{1/q}\\
&\quad+N\kappa^{-d/q}
\left(\fint_{B_{r/2}(x_{0})}|D_{x'}v|^{q}+|V|^{q}\ dx\right)^{1/q}\\
&\leq N_{0}\kappa\left(\fint_{B_{r}(x_{0})}|(D_{x'}u,U)-\mathbf{q}|^{q}\ dx\right)^{1/q}+N\kappa^{-d/q}\Big(\|Du\|_{L^{\infty}(B_{r}(x_{0}))}\\
&\quad\cdot\fint_{B_{r}(x_{0})}|A(x)-\bar{A}(x'_{0},x^{d})|\ dx+\fint_{B_{r}(x_{0})}|g_{d}(x)-\bar{g}_{d}(x'_{0},x^{d})|\ dx\Big)\\
&\quad+N\kappa^{-d/q}
\left(\fint_{B_{r/2}(x_{0})}|D_{x'}v|^{q}+|V|^{q}\ dx\right)^{1/q}\\
&\leq N_{0}\kappa\left(\fint_{B_{r}(x_{0})}|(D_{x'}u,U)-\mathbf{q}|^{q}\ dx\right)^{1/q}+N\kappa^{-d/q}\Big(\|Du\|_{L^{\infty}(B_{r}(x_{0}))}\bar\omega_{A}(r)+\bar\omega_{g}(r)\Big).
\end{align*}
Since $\mathbf{q}\in\mathbb R^{n\times d}$ is arbitrary, we obtain
\begin{align*}
\phi(x_{0},\kappa r)\leq N_{0}\kappa\phi(x_{0},r)+N\kappa^{-d/q}\Big(\|Du\|_{L^{\infty}(B_{r}(x_{0}))}\bar\omega_{A}(r)+\bar\omega_{g}(r)\Big).
\end{align*}
For any given $\gamma\in(0,1)$, fix a $\kappa\in(0,1/2)$ small enough so that $N_{0}\kappa\leq\kappa^{\gamma}$. Therefore, we have
\begin{align*}
\phi(x_{0},\kappa r)\leq \kappa^{\gamma}\phi(x_{0},r)+N\Big(\|Du\|_{L^{\infty}(B_{r}(x_{0}))}\bar\omega_{A}(r)+\bar\omega_{g}(r)\Big).
\end{align*}
Using the fact that $\kappa^{\gamma}<1$, by iteration, for $j=1,2,\ldots$, we obtain
\begin{align*}
\phi(x_{0},\kappa^{j}r)
&\leq\kappa^{j\gamma}\phi(x_{0},r)\\
&\quad+N\left(\|Du\|_{L^{\infty}(B_{r}(x_{0}))}\sum_{i=1}^{j}\kappa^{(i-1)\gamma}\bar\omega_{A}(\kappa^{j-i}r)+\sum_{i=1}^{j}\kappa^{(i-1)\gamma}\bar\omega_{g}(\kappa^{j-i}r)\right).
\end{align*}
Therefore, we get
\begin{align}\label{iteration phi}
\phi(x_{0},\kappa^{j}r)\leq\kappa^{j\gamma}\phi(x_{0},r)+N\|Du\|_{L^{\infty}(B_{r}(x_{0}))}\tilde\omega_{A}(\kappa^{j}r)+N\tilde\omega_{g}(\kappa^{j}r),
\end{align}
where
\begin{align}\label{tilde phi}
\tilde\omega_{\bullet}(t)=\sum_{i=1}^{\infty}\kappa^{i\gamma}\Big(\bar\omega_{\bullet}(\kappa^{-i}t)\chi_{\kappa^{-i}t\leq1}+\bar\omega_{\bullet}(1)\chi_{\kappa^{-i}t>1}\Big).
\end{align}
Moreover, $\tilde\omega_{\bullet}(t)$ is a Dini function (see Lemma 1 in \cite{d}) and satisfies \eqref{equivalence}.

Now, for any $\rho$ satisfying $0<\rho\leq r\leq 1/4$, we take $j$ to be the integer satisfying $\kappa^{j+1}<{\rho}/{r}\leq\kappa^{j}$. Then, by \eqref{iteration phi} and \eqref{equivalence}, we have
\begin{align}\label{itera phi}
\phi(x_{0},\rho)\leq N\Big(\frac{\rho}{r}\Big)^{\gamma}\phi(x_{0},r)+N\tilde{\omega}_{A}(\rho)\|Du\|_{L^{\infty}(B_{r}(x_{0}))}+N\tilde{\omega}_{g}(\rho).
\end{align}
Hence, \eqref{est phi'} follows from \eqref{est phi} and \eqref{itera phi}.
\end{proof}

\begin{lemma}
                    \label{lem3.4}
We have
\begin{align}\label{est Du''}
\|Du\|_{L^{\infty}(B_{1/4})}\leq N\|(D_{x'}u,U)\|_{L^{1}(B_{3/4})}+N\left(\int_{0}^{1}\frac{\tilde\omega_{g}(t)}{t}\ dt+\|g\|_{L^{\infty}(\cD)}\right),
\end{align}
where $N>0$ is a \it{universal constant}.
\end{lemma}

\begin{proof}
Take $\mathbf q_{x_{0},r}\in\mathbb R^{n\times d}$ to be such that
$$\phi(x_{0},r)=\left(\fint_{B_{r}(x_{0})}|(D_{x'}u,U)-\mathbf q_{x_{0},r}|^{q}\ dx\right)^{1/q}.$$
Similarly, we find $\mathbf q_{x_{0},\kappa r}\in\mathbb R^{n\times d}$, et cetera. Notice that
$$|\mathbf q_{x_{0},\kappa r}-\mathbf q_{x_{0},r}|^{q}\leq|(D_{x'}u,U)-\mathbf q_{x_{0},r}|^{q}+|(D_{x'}u,U)-\mathbf q_{x_{0},\kappa r}|^{q}.$$
By taking average over $x\in B_{\kappa r}(x_{0})$ and taking the $q$-th root, we obtain
$$|\mathbf q_{x_{0},\kappa r}-\mathbf q_{x_{0},r}|\leq N(\phi(x_{0},\kappa r)+\phi(x_{0},r)).$$
By iteration, we have
\begin{align}\label{mathbf q}
|\mathbf q_{x_{0},\kappa^{K}r}-\mathbf q_{x_{0},r}|\leq N\sum_{j=0}^{K}\phi(x_{0},\kappa^{j}r).
\end{align}
Clearly, \eqref{iteration phi} implies that
$$\lim_{K\rightarrow\infty}\phi(x_{0},\kappa^{K}r)=0.$$
Thus, by using the assumption that $u\in C^{0,1}(B_{3/4})$, we obtain for {\em a.e.} $x_{0}\in B_{3/4}$,
$$\lim_{K\rightarrow\infty}\mathbf q_{x_{0},\kappa^{K}r}=(D_{x'}u(x_{0}),U(x_{0})).$$
On the other hand, recalling that $\tilde\omega_{A}$ and $\tilde\omega_{g}$ satisfy \eqref{equivalence}. Therefore, by taking $K\rightarrow\infty$ in \eqref{mathbf q}, using \eqref{iteration phi} and Lemma \ref{lemma omiga}, for {\em a.e.} $x_{0}\in B_{3/4}$, we have
\begin{align}\label{est Du q}
&|(D_{x'}u(x_{0}),U(x_{0}))-\mathbf q_{x_{0},r}|\nonumber\\
&\leq N\sum_{j=0}^{\infty}\phi(x_{0},\kappa^{j}r)\nonumber\\
&\leq N\left(\phi(x_{0},r)+\|Du\|_{L^{\infty}(B_{r}(x_{0}))}\int_{0}^{r}\frac{\tilde\omega_{A}(t)}{t}\ dt+\int_{0}^{r}\frac{\tilde\omega_{g}(t)}{t}\ dt\right).
\end{align}
By averaging the inequality
$$|\mathbf q_{x_{0},r}|^{q}\leq|(D_{x'}u,U)-\mathbf q_{x_{0},r}|^{q}+|(D_{x'}u,U)|^{q}$$
over $x\in B_{r}(x_{0})$ and taking the $q$-th root, we have
$$|\mathbf q_{x_{0},r}|\leq2^{1/q-1}\phi(x_{0},r)+2^{1/q-1}\left(\fint_{B_{r}(x_{0})}|(D_{x'}u,U)|^{q}\ dx\right)^{1/q}.$$
Therefore, combining \eqref{est Du q} and \eqref{est phi}, we obtain for {\em a.e.} $x_{0}\in B_{3/4}$,
\begin{align*}
|(D_{x'}u(x_{0}),U(x_{0}))|&\leq Nr^{-d}\|(D_{x'}u,U)\|_{L^{1}(B_{r}(x_{0}))}\\
&\quad+N\left(\|Du\|_{L^{\infty}(B_{r}(x_{0}))}\int_{0}^{r}\frac{\tilde\omega_{A}(t)}{t}\ dt+\int_{0}^{r}\frac{\tilde\omega_{g}(t)}{t}\ dt\right).
\end{align*}
For any $x_{1}\in B_{1/4}$ and $0<r<1/4$, we take the supremum over $B_{r}(x_{1})$ and use $A^{dd}\geq\nu$ to obtain
\begin{align*}
\|Du\|_{L^{\infty}(B_{r}(x_{1}))}&\leq Nr^{-d}\|(D_{x'}u,U)\|_{L^{1}(B_{2r}(x_{1}))}+N\Big(\int_{0}^{r}\frac{\tilde\omega_{g}(t)}{t}\ dt\\
&\quad+\|Du\|_{L^{\infty}(B_{2r}(x_{1}))}\int_{0}^{r}\frac{\tilde\omega_{A}(t)}{t}\ dt+\|g\|_{L^{\infty}(\cD)}\Big).
\end{align*}
We fix $r_{0}<1/4$ such that for any $0<r\leq r_{0}$, we have
$$N\int_{0}^{r}\frac{\tilde\omega_{A}(t)}{t}\ dt\leq 4^{-d}.$$
Then, for any $x_{1}\in B_{1/4}$ and $0<r\leq r_{0}$, we have
\begin{align*}
\|Du\|_{L^{\infty}(B_{r}(x_{1}))}&\leq4^{-d}\|Du\|_{L^{\infty}(B_{2r}(x_{1}))}
+Nr^{-d}\|(D_{x'}u,U)\|_{L^{1}(B_{2r}(x_{1}))}\\
&\quad+N\left(\int_{0}^{r}\frac{\tilde\omega_{g}(t)}{t}\ dt+\|g\|_{L^{\infty}(\cD)}\right).
\end{align*}
For $k=1,2,\ldots$, denote $r_{k}={3}/{4}-(1/2)^{k}$. For $x_{1}\in B_{r_{k}}$ and $r=(1/2)^{k+2}$, we have $B_{2r}(x_{1})\subset B_{r_{k+1}}$. We take $k_{0}\geq1$ large enough such that $(1/2)^{k_{0}+2}\leq r_{0}$. It follows that for any $k\geq k_{0}$,
\begin{align*}
\|Du\|_{L^{\infty}(B_{r_{k}})}&\leq4^{-d}\|Du\|_{L^{\infty}(B_{r_{k+1}})}
+N2^{kd}\|(D_{x'}u,U)\|_{L^{1}(B_{3/4})}\\
&\quad+N\left(\int_{0}^{1}\frac{\tilde\omega_{g}(t)}{t}\ dt+\|g\|_{L^{\infty}(\cD)}\right).
\end{align*}
By multiplying the above by $4^{-kd}$ and summing over $k=k_{0},k_{0}+1,\ldots$, we have
\begin{align*}
&\sum_{k=k_{0}}^{\infty}4^{-kd}\|Du\|_{L^{\infty}(B_{r_{k}})}\\
&\leq\sum_{k=k_{0}}^{\infty}4^{-(k+1)d}\|Du\|_{L^{\infty}(B_{r_{k+1}})}
+N\|(D_{x'}u,U)\|_{L^{1}(B_{3/4})}+N\left(\int_{0}^{1}
\frac{\tilde\omega_{g}(t)}{t}\ dt+\|g\|_{L^{\infty}(\cD)}\right).
\end{align*}
It follows from $u\in C^{0,1}(B_{3/4})$ that the summations on both sides are convergent. We thus obtain
\begin{align*}
\|Du\|_{L^{\infty}(B_{1/4})}\leq N\|(D_{x'}u,U)\|_{L^{1}(B_{3/4})}+N\left(\int_{0}^{1}\frac{\tilde\omega_{g}(t)}{t}\ dt+\|g\|_{L^{\infty}(\cD)}\right).
\end{align*}
The lemma is proved.
\end{proof}

Finally, we are ready to prove Proposition \ref{main prop}.

\begin{proof}[\bf Proof of Proposition \ref{main prop}.]
By \eqref{est Du q}, we have for $0<r<1/8$ that
\begin{align*}
&\sup_{x_{0}\in B_{1/8}}|(D_{x'}u(x_{0}),U(x_{0}))-\mathbf q_{x_{0},r}|\\
&\leq N\sup_{x_{0}\in B_{1/8}}\phi(x_{0},r)+N\|Du\|_{L^{\infty}(B_{1/4})}\int_{0}^{r}\frac{\tilde\omega_{A}(t)}{t}\ dt+N\int_{0}^{r}\frac{\tilde\omega_{g}(t)}{t}\ dt\\
&=:N\psi(r).
\end{align*}
We recall that for each $x_0$, the coordinate system and thus $x'$ are chosen according to $x_0$. By Lemma \ref{lemma itera}, for any $0<r<1/8$, we obtain
\begin{align}\label{sup phi}
\sup_{x_{0}\in B_{1/8}}\phi(x_{0},r)\leq N\left(r^{\gamma}\|(D_{x'}u,U)\|_{L^{1}(B_{1/4})}+\tilde{\omega}_{A}(r)\|Du\|_{L^{\infty}(B_{1/4})}+\tilde{\omega}_{g}(r)\right).
\end{align}
Suppose that $y\in B_{1/8}\cap \cD_{j_{1}}$ for some $j_{1}\in[1,l+1]$. If $|x_{0}-y|\geq 1/32$, combining $$|(D_{x'}u(x_{0}),U(x))-(D_{x'}u(y),U(y))|\leq2\big(\|Du\|_{L^{\infty}(B_{1/4})}+\|g\|_{L^{\infty}(\cD)}\big)$$
and \eqref{est Du''}, we have
\begin{align}\label{C1 est2}
&|(D_{x'}u(x_{0}),U(x_{0}))-(D_{x'}u(y),U(y))|\nonumber\\
&\leq
N|x_{0}-y|^{\gamma}\left( \|(D_{x'}u,U)\|_{L^{1}(B_{3/4})}+\int_{0}^{1}\frac{\tilde\omega_{g}(t)}{t}\ dt+\|g\|_{L^{\infty}(\cD)}\right),
\end{align}
where $\gamma\in(0,1)$ is a constant. On the other hand, if $|x_{0}-y|<1/32$, we set $r=|x_{0}-y|$ and discuss it further according to the following dichotomy.

{\bf Case 1.} If
$$
r\leq 1/16\max\{\mbox{dist}(x_{0},\partial \cD_{j_{0}}),\mbox{dist}(y,\partial \cD_{j_{1}})\},
$$
then $j_{0}=j_{1}$. By using the triangle inequality, we have
\begin{align}\label{case1}
&|(D_{x'}u(x_{0}),U(x_{0}))-(D_{x'}u(y),U(y))|^{q}\nonumber\\
&\leq|(D_{x'}u(x_{0}),U(x_{0}))-\mathbf q_{x_{0},r}|^{q}+|\mathbf q_{x_{0},r}-\mathbf q_{y,r}|^{q}+|(D_{x'}u(y),U(y))-\mathbf q_{y,r}|^{q}\nonumber\\
&\leq N\psi^{q}(r)+|(D_{x'}u(z),U(z))-\mathbf q_{x_{0},r}|^{q}+|(D_{x'}u(z),U(z))-\mathbf q_{y,r}|^{q}\nonumber\\
&\quad+|(D_{x'}u(y),U(y))-\mathbf q_{y,r}|^{q},\quad\forall~z\in B_{r}(x_{0})\cap B_{r}(y).
\end{align}
In order to estimate the last two terms, we define
$$\varphi(y,r):=\inf_{\mathbf q\in\mathbb R^{n\times d}}\left(\fint_{B_{r}(y)}|Du-\mathbf q|^{q}\ dx\right)^{1/q}.$$
We use $D_{y}$ to denote the derivative in the coordinate system associated with $y$. Then for any $\mathbf q=(\mathbf q',q_{d})$, we have
\begin{align}\label{identity1}
&(D_{y'}u-\mathbf q',\hat{A}^{d\beta}D_{\beta}u-\hat{g}_{d}-q_{d})\nonumber\\
&=\Big(D_{y'}u-\mathbf q',D_{y^{d}}u-(\hat{A}^{dd})^{-1}\big(\hat{g}_{d}
+q_{d}-\sum_{\beta=1}^{d-1}\hat{A}^{d\beta}q_{\beta}\big)\Big)E,
\end{align}
where $\hat{A}^{d\beta}$ and $\hat{g}_{d}$ are constants corresponding to $A^{d\beta}$ and $g_{d}$, respectively, and $E=(E^{\alpha\beta})$ is defined by
\begin{equation*}
\begin{split}
&E^{\alpha\beta}=\delta_{\alpha\beta}~ \mbox{for}~\alpha, \beta\in\{1,\ldots,d-1\};\quad E^{d\beta}=0~ \mbox{for}~\beta\in\{1,\ldots,d-1\};\\
&E^{\alpha d}=\hat{A}^{d\alpha}~\mbox{for}~ \alpha\in\{1,\ldots,d\}.
\end{split}
\end{equation*}
From \eqref{identity1}, we get
\begin{align}\label{varphy1}
&\varphi(y,r)\leq \left(\fint_{B_{r}(y)}\Big|\big(D_{y'}u-\mathbf q',D_{y^{d}}u-(\hat{A}^{dd})^{-1}\big(\hat{g}_{d}+q_{d}
-\sum_{\beta=1}^{d-1}\hat{A}^{d\beta}q_{\beta}\big)\big)\Big|^{q}\ dx\right)^{1/q}\nonumber\\
&=\left(\fint_{B_{r}(y)}|(D_{y'}u-\mathbf q',\hat{A}^{d\beta}D_{\beta}u-\hat{g}_{d}-q_{d})E^{-1}|^{q}\ dx\right)^{1/q}\nonumber\\
&=\left(\fint_{B_{r}(y)}\Big|\big((D_{y'}u-\mathbf q',U-q_{d})+(0',\hat{A}^{d\beta}D_{\beta}u-A^{d\beta}D_{\beta}u
+g_{d}-\hat{g}_{d})\big)E^{-1}\Big|^{q}\ dx\right)^{1/q}\nonumber\\
&\leq2^{1/q-1}\left(\fint_{B_{r}(y)}\Big|\big((D_{y'}u,U)-\mathbf q\big)E^{-1}\Big|^{q}\right)^{1/q}+N\left(\fint_{B_{r}(y)}|(0',g_{d}
-\hat{g}_{d})E^{-1}|^{q}\ dx\right)^{1/q}\nonumber\\
&\quad+N\left(\fint_{B_{r}(y)}|(0',A^{d\beta}D_{\beta}u
-\hat{A}^{d\beta}D_{\beta}u)E^{-1}|^{q}\ dx\right)^{1/q}\nonumber\\
&\leq N\left(\fint_{B_{r}(y)}|(D_{y'}u,U)-\mathbf q|^{q}\right)^{1/q}+N\left(\omega_{A}(r)\|Du\|_{L^{\infty}(B_{1/4})}
+\omega_{g}(r)\right).
\end{align}
Since $\mathbf q$ is arbitrary, we obtain
\begin{align}\label{varphy2}
\varphi(y,r)\leq N\left(\phi(y,r)+\omega_{A}(r)\|Du\|_{L^{\infty}(B_{1/4})}+\omega_{g}(r)\right).
\end{align}
In the coordinate system associated with $x_{0}$, we first notice that
\begin{align}\label{identity2}
(D_{x'}u,D_{x^{d}}u)=(D_{y'}u,D_{y^{d}}u)X,
\end{align}
where $X=(X^{\alpha\beta})$ is a $d\times d$ matrix, and
\begin{equation}\label{def X}
X^{\alpha\beta}=\frac{\partial y^{\alpha}}{\partial x^{\beta}} \,\,~\mbox{for} ~\alpha,\beta=1,\dots,d.
\end{equation}
By using \eqref{identity2}, we obtain
\begin{align*}
\varphi(y,r)
&\leq \left(\fint_{B_{r}(y)}\Big|Du-\mathbf q X\Big|^{q}\ dx\right)^{1/q}\\
&=\left(\fint_{B_{r}(y)}|\big(D_{y'}u-\mathbf q',D_{y^{d}}u-q_{d}\big)X|^{q}\ dx\right)^{1/q}.
\end{align*}
Then by using \eqref{varphy1}, \eqref{varphy2}, and the fact that $\mathbf q$ is arbitrary, we have
\begin{equation}\label{est varphi}
\varphi(y,r)\leq N\left(\phi(y,r)+\omega_{A}(r)\|Du\|_{L^{\infty}(B_{1/4})}+\omega_{g}(r)\right)
\end{equation}
in the coordinate system associated with $x_{0}$. We thus have proved that the upper bound of $\varphi(y,r)$ is independent of coordinate systems. Now, we denote
$$\phi_{x_{0}}(y,r):=\inf_{\mathbf q\in\mathbb R^{n\times d}}\left(\fint_{B_{r}(y)}|(D_{x'}u,U)-\mathbf q|^{q}\ dx\right)^{1/q}.$$
Then,
\begin{align*}
&\phi_{x_{0}}(y,r)\leq\left(\fint_{B_{r}(y)}\Big|\big(D_{x'}u-\mathbf q',U-\big(\hat{A}^{d\beta}q_{\beta}-\hat{g}_{d}\big)\big)\Big|^{q}\ dx\right)^{1/q}\\
&=\left(\fint_{B_{r}(y)}|(D_{x'}u-\mathbf q',\hat{A}^{d\beta}D_{\beta}u-\hat{A}^{d\beta}q_{\beta})
+(0',U-\hat{A}^{d\beta}D_{\beta}u+\hat{g}_{d})|^{q}\ dx\right)^{1/q}\\
&=\left(\fint_{B_{r}(y)}|(D_{x'}u-\mathbf q',D_{x^{d}}u-q_{d})E+(0',A^{d\beta}D_{\beta}u
-\hat{A}^{d\beta}D_{\beta}u+\hat{g}_{d}-g_{d})|^{q}\ dx\right)^{1/q}.
\end{align*}
By \eqref{est varphi}, we have
\begin{align*}
\phi_{x_{0}}(y,r)\leq N\left(\phi(y,r)+\omega_{A}(r)\|Du\|_{L^{\infty}(B_{1/4})}+\omega_{g}(r)\right).
\end{align*}
Therefore, by using a similar argument that led to \eqref{est Du q}, we get
\begin{align*}
&|(D_{x'}u(y),U(y))-\mathbf q_{y,r}|\\
&\leq N \left(\phi(y,r)+\|Du\|_{L^{\infty}(B_{1/4})}
\int_{0}^{r}\frac{\tilde\omega_{A}(t)}{t}\ dt+\int_{0}^{r}\frac{\tilde\omega_{g}(t)}{t}\ dt\right).
\end{align*}
Now, coming back to \eqref{case1}, taking the average over $z\in B_{r}(x_{0})\cap B_{r}(y)$, and then taking the $q$-th root, we get
\begin{align*}
&|(D_{x'}u(x_{0}),U(x_{0}))-(D_{x'}u(y),U(y))|\\
&\leq N\Big(\psi(r)+\phi(x_{0},r)+\phi(y,r)\Big)\leq N\psi(r).
\end{align*}
Therefore, it follows from \eqref{est Du''} and \eqref{sup phi} that
\begin{align}\label{C1 est}
&|(D_{x'}u(x_{0}),U(x_{0}))-(D_{x'}u(y),U(y))|\nonumber\\
&\leq N|x_{0}-y|^{\gamma}\|(D_{x'}u,U)\|_{L^{1}(B_{3/4})}+N\int_{0}^{|x_{0}-y|}\frac{\tilde{\omega}_{g}(t)}{t}\ dt\nonumber\\
&\quad+N\int_{0}^{|x_{0}-y|}\frac{\tilde{\omega}_{A}(t)}{t}\ dt
\left(\|(D_{x'}u,U)\|_{L^{1}(B_{3/4})}+\int_{0}^{1}\frac{\tilde{\omega}_{g}(t)}{t}\ dt+\|g\|_{L^{\infty}(\cD)}\right).
\end{align}

{\bf Case 2.} If $r>1/16\max\{\mbox{dist}(x_{0},\partial \cD_{j_{0}}),\mbox{dist}(y,\partial \cD_{j_{1}})\}$, then
\begin{align}\label{case2}
&|(D_{x'}u(x_{0}),U(x_{0}))-(D_{x'}u(y),U(y))|^{q}\nonumber\\
&\leq|(D_{x'}u(x_{0}),U(x_{0}))-\mathbf q_{x_{0},r}|^{q}+|\mathbf q_{x_{0},r}-\mathbf q_{y,r}|^{q}+|(D_{y'}u(y),U(y))-\mathbf q_{y,r}|^{q}\nonumber\\
&\quad+|(D_{y'}u(y),U(y))-(D_{x'}u(y),U(y))|^{q}\nonumber\\
&\leq N\psi^{q}(r)+|(D_{x'}u(z),U(z))-\mathbf q_{x_{0},r}|^{q}+|(D_{y'}u(z),U(z))-\mathbf q_{y,r}|^{q}\nonumber\\
&\quad+|(D_{y'}u(z),U(z))-(D_{x'}u(z),U(z))|^{q}\nonumber\\
&\quad+|(D_{y'}u(y),U(y))-(D_{x'}u(y),U(y))|^{q},\quad\forall~z\in B_{r}(x_{0})\cap B_{r}(y).
\end{align}
For the last term, on one hand, $$D_{x'}u(y)-D_{y'}u(y)=(D_{x'}u(y),D_{x^{d}}u(y))(I-X^{-1})I_{0},$$ where $I_{0}=(I^{\alpha\beta})$ is a $d\times (d-1)$ matrix with
\begin{equation*}
I^{\alpha\beta}=\delta_{\alpha\beta}~\mbox{for}~\alpha,\beta\in\{1,\dots,d-1\};\quad I^{d\beta}=0~\mbox{for}~\beta\in\{1,\dots,d-1\},
\end{equation*}
$X$ is defined by \eqref{def X}, and $I$ is a $d\times d$ identity matrix. On the other hand, we suppose that the closest point on $\partial D_{j_{1}}$ to $y$ is $(y',h_{j_{1}}(y'))$, and let $$n_{2}=\frac{\big(-\nabla_{x'}h_{j_{1}}(y'),1\big)^{\top}}{\sqrt{1+|\nabla_{x'}h_{j_{1}}(y')|^{2}}}$$
be the unit normal vector at $(y',h_{j_{1}}(y'))$ on the surface $\{(y',t): t=h_{j_{1}}(y')\}$. The corresponding tangential vectors are
\begin{align*}
\tau_{2,1}=(1,0,\ldots,0,D_{x^{1}}h_{j_{1}}(y'))^{\top},\quad\dots,\quad
\tau_{2,d-1}=(0,0,\ldots,1,D_{x^{d-1}}h_{j_{1}}(y'))^{\top}.
\end{align*}
To make them orthogonal to each other, we define the projection operator by
$$\mbox{proj}_{a}b=\frac{\langle a,b\rangle}{\langle a,a\rangle}a,$$
where $\langle a,b\rangle$ denotes the inner product of the vectors $a$ and $b$, and $\langle a,a\rangle=\|a\|^{2}$. Then the Gram-Schmidt process works as follows:
\begin{align*}
\hat{\tau}_{2,1}&=\tau_{2,1},\quad\tilde{\tau}_{2,1}=\frac{\hat{\tau}_{2,1}}{\|\hat{\tau}_{2,1}\|},\\
\hat{\tau}_{2,2}&=\tau_{2,2}-\mbox{proj}_{\hat{\tau}_{2,1}}\tau_{2,2},\quad\tilde{\tau}_{2,2}=\frac{\hat{\tau}_{2,2}}{\|\hat{\tau}_{2,2}\|},\\
&\vdots\\
\hat{\tau}_{2,d-1}&=\tau_{2,d-1}-\sum_{j=1}^{d-2}\mbox{proj}_{\hat{\tau}_{2,j}}\tau_{2,d-1},\quad\tilde{\tau}_{2,d-1}=\frac{\hat{\tau}_{2,d-1}}{\|\hat{\tau}_{2,d-1}\|}.
\end{align*}
Similarly, we use $n_{1}=\frac{\big(-\nabla_{x'}h_{j_{0}}(x'_{0}),1\big)^{\top}}{\sqrt{1+|\nabla_{x'}h_{j_{0}}(x'_{0})|^{2}}}=(0',1)^{\top}$ to denote the unit normal vector at $(x'_{0},h_{j_{0}}(x'_{0}))$, and the corresponding tangential vectors are
\begin{align*}
\tau_{1,1}=(1,0,\ldots,0,0)^{\top},\quad\ldots,\quad
\tau_{1,d-1}=(0,0,\ldots,1,0)^{\top}.
\end{align*}
It follows from the proof of Lemma \ref{volume} that the upper bound of $|\nabla_{x'}h_{j_1}(y')|$ is $N\omega_1(r)$. Then we have
\begin{align*}
|n_{1}-n_{2}|=\left|(0',1)^{\top}-\frac{\big(-\nabla_{x'}h_{j_{1}}(y'),1\big)^{\top}}{\sqrt{1+|\nabla_{x'}h_{j_{1}}(y')|^{2}}}\right|
&\leq N\omega_1(|x_{0}-y|),
\end{align*}
which is also true for $|\tau_{1,i}-\tilde{\tau}_{2,i}|, i=1,\ldots,d-1$. Thus, we obtain
\begin{align*}
|D_{x'}u(y)-D_{y'}u(y)|\leq N\|Du\|_{L^{\infty}(B_{1/4})}\omega_1(|x_{0}-y|).
\end{align*}
Similarly, we can estimate the difference of $U$ in different coordinate systems. Therefore, we obtain
\begin{align}\label{diffe coor}
|(D_{x'}u(y),U(y))-(D_{y'}u(y),U(y))|\leq N\|Du\|_{L^{\infty}(B_{1/4})}\omega_1(|x_{0}-y|).
\end{align}
We remark that the penultimate term of \eqref{case2} also satisfies \eqref{diffe coor}. Coming back to \eqref{case2}, we take the average over $z\in B_{r}(x_{0})\cap B_{r}(y)$ and take the $q$-th root to get
\begin{align*}
&|(D_{x'}u(x_{0}),U(x_{0}))-(D_{x'}u(y),U(y))|\nonumber\\
&\leq N\Big(\psi(r)+\phi(x_{0},r)+\phi(y,r)+\|Du\|_{L^{\infty}(B_{1/4})}\omega_1(|x_{0}-y|)\Big)\nonumber\\
&\leq N\Big(\psi(r)+\|Du\|_{L^{\infty}(B_{1/4})}\omega_1(|x_{0}-y|)\Big).
\end{align*}
Therefore, it follows from \eqref{est Du''} and \eqref{sup phi} that
\begin{align}\label{C1 est1}
&|(D_{x'}u(x_{0}),U(x_{0}))-(D_{x'}u(y),U(y))|\nonumber\\
&\leq N|x_{0}-y|^{\gamma}\|(D_{x'}u,U)\|_{L^{1}(B_{3/4})}+N\int_{0}^{|x_{0}-y|}\frac{\tilde{\omega}_{g}(t)}{t}\ dt\nonumber\\
&\quad+N\int_{0}^{|x_{0}-y|}\frac{\tilde{\omega}_{A}(t)}{t}\ dt\cdot\left(\|(D_{x'}u,U)\|_{L^{1}(B_{3/4})}+\int_{0}^{1}\frac{\tilde{\omega}_{g}(t)}{t}\ dt+\|g\|_{L^{\infty}(D)}\right).
\end{align}
Thus, Proposition \ref{main prop} is proved under the assumption that $u\in C^{0,1}(B_{3/4})$.

We now show that $u\in C^{0,1}(B_{3/4})$ by using the technology of locally flattening the boundaries and an approximation argument. By the interior regularity obtained in \cite{dk}, it suffices to show that for any $x_{0}\in\partial\cD_{j}, j=1,\ldots,M-1$, there is a neighborhood of $x_0$ in which $u$ is Lipschitz. Recall that $x_0$ belongs to the boundaries of at most two of the subdomains. Thus, we can find a small $r_0>0$ and a $C^{1,\text{Dini}}$ diffeomorphism to flatten the boundary $\partial\cD_{j}\cap B_{r_0}(x_{0})$:
$$y=\Phi(x)=(\Phi^{1}(x),\dots,\Phi^{d}(x)),$$
which satisfies $\Phi(x_0)=0$, $\det D\Phi=1$, and
$$
\Phi(\partial\cD_{j}\cap B_{r_0}(x_{0}))=\Phi(B_{r_0}(x_{0}))\cap \{y^d=0\}.
$$
Then $\hat{u}(y):=u(x)$ satisfies
$$D_{\alpha}(\hat{A}^{\alpha\beta}D_{\beta}\hat{u})=\Div\hat{g},$$
where $\hat{A}^{\alpha\beta}(y)=D_{k}\Phi^{\alpha}D_{l}\Phi^{\beta}A^{kl}(x)$ and $\hat g(y)=D\Phi^{\top}g(x)$. Note that the coefficients $\hat{A}^{\alpha\beta}$ and $\hat{g}$ are also of piecewise Dini mean oscillation in $\Phi(B_{r_0}(x_0))$. To show that $u$ is Lipschitz near $x_0$, we only need to show that $\hat u$ is Lipschitz near $0$. Now we take the standard mollification of the coefficients and data in the $y'$ direction with a parameter $\varepsilon>0$, apply the result in \cite{d} as well as the a priori Lipschitz estimate in Lemma \ref{lem3.4} to get a uniform Lipschitz estimate independent of $\varepsilon$, and finally take the limit as $\varepsilon \searrow 0$ by following the proof of \cite[Theorem 1]{d}.
Theorem \ref{thm1} is proved.
\end{proof}

\subsection{Proof of Corollary \ref{coro}}

Similar to the proof of Theorem \ref{thm1}, we take $x_{0}\in B_{3/4}\cap \cD_{j_{0}}$. Let $A^{(j)}\in C^{\delta}(\overline{\cD}_{j})$, $1\leq j\leq l+1$, be matrix-valued functions, $B^{(j)}, g^{(j)}$ be in $C^{\delta}(\overline{\cD}_{j})$. Define the piecewise constant (matrix-valued) functions
\begin{align*}
\bar{A}(x)=
A^{(j)}(x'_{0},h_{j}(x'_{0})),\quad x\in\Omega_{j}.
\end{align*}
From $B^{(j)}$ and $g^{(j)}$, we similarly define piecewise constant functions $\bar{B}$ and $\bar{g}$. Using Remark \ref{rmk volume}, we get the following result, which is similar to Lemma 5.2 in \cite{lv}.
\begin{lemma}
Let $$1<\delta'=\min\{\delta,\frac{\mu}{1+\mu}\}.$$
With $A, \bar{A}, B, \bar{B}, g$, and $\bar{g}$ be defined as above, there exists a positive constant $N$, depending only on $d,l,\mu,\delta,\nu,\Lambda$, $\max_{1\leq j\leq l+1}\|A\|_{C^{\delta}(\overline{D}_{j})}$, $\max_{1\leq j\leq l+1}\|B\|_{C^{\delta}(\overline{D}_{j})}$, $\max_{1\leq j\leq l+1}\|g\|_{C^{\delta}(\overline{D}_{j})}$ and $\max_{1\leq j\leq l+1}\|h_{j}\|_{C^{1,\mu}(\overline{D}_{j})}$, such that for $0<r\leq 1$,
\begin{align*}
\fint_{B_{r}(x_{0})}|A-\bar{A}|\ dx+\fint_{B_{r}(x_{0})}|B-\bar{B}|\ dx
+\fint_{B_{r}(x_{0})}|g-\bar{g}|\ dx\leq Nr^{\delta'}.
\end{align*}
\end{lemma}
Thus, Corollary \ref{coro} directly follows from \eqref{C1 est}, \eqref{C1 est1}, and \eqref{C1 est2} by taking $\gamma\in(\delta',1)$.

\section{Proof of Corollary \ref{coro2}}\label{sec_coro2}

We shall make use of the idea in \cite{a,b}, where the $W_{\text{loc}}^{1,p}$-regularity was proved for $W^{1,1}$ weak solutions to divergence form elliptic equations with Dini continuous coefficients by using a duality argument, $L^{p}$-regularity property, and bootstrap arguments. In our case, we will use the $W^{1,p}$ estimate in Lemma \ref{lem loc lq} and the interior $W^{1,\infty}$-regularity obtained in Theorem \ref{thm1}.

\begin{proof}[\bf Proof of Corollary \ref{coro2}.]
By the Sobolev embedding theorem, we have $u\in L^{\frac{d}{d-1}}(\cD)$. Thus, we only need to prove that $Du\in L_{\text{loc}}^{p}(\cD)$ for some $p\in(1,\frac{d}{d-1})$. We fix some $1<p<\frac{d}{d-1}$ so that $2\leq d<p'<\infty$ with ${1}/{p}+{1}/{p'}=1$. We rewrite \eqref{systems} as
\begin{equation*}
\mathcal{L}'u:=D_{\alpha}(A^{\alpha\beta}D_{\beta}u)+D_{\alpha}(B^{\alpha}u)+\hat{B}^{\alpha}D_{\alpha}u+(C-\lambda_{0})u=\Div g+f-\lambda_{0}u,
\end{equation*}
where $\lambda_{0}$ is a fixed large enough number. Denote $f_{0}:=f-\lambda_{0}u\in L^{\frac{d}{d-1}}(\cD)$. Let $h\in C_{0}^{\infty}(\cD)$ be given, and $v\in H_{0}^{1}(\cD)$ be the solution of
$$\mathcal{L}'^{*}v:=D_{\beta}(A^{\alpha\beta}D_{\alpha}v)-D_{\alpha}(\hat{B}^{\alpha}v)-B^{\alpha}D_{\alpha}v+(C-\lambda_{0})v=\Div h,$$
where $\mathcal{L}'^{*}$ is the adjoint operator of $\mathcal{L}'$. Then, by Theorem \ref{thm1}, we obtain $Dv\in L^{\infty}(\cD_{\varepsilon})$. By using the definition of weak solutions, uniform ellipticity condition, H\"{o}lder's inequality, $p'>2$, and the fact that $\lambda_{0}$ is a large enough number, we get
\begin{align}\label{bound v}
\|v\|_{H^{1}(\cD)}\leq N\|h\|_{L^{2}(\cD)}\leq N\|h\|_{L^{p'}(\cD)}.
\end{align}
By Lemma \ref{lem loc lq} and \eqref{bound v}, we obtain
\begin{align}\label{W1p' v}
\|v\|_{W^{1,p'}(\cD_{\varepsilon})}\leq N\big(\|h\|_{L^{p'}(\cD)}+\|v\|_{L^{2}(\cD)}\big)\leq N\|h\|_{L^{p'}(\cD)}.
\end{align}
Since $p'>d$, it follows from Morrey's inequality that
\begin{equation}\label{v infty}
\|v\|_{L^{\infty}(\cD_{\varepsilon})}\leq N\|h\|_{L^{p'}(\cD)}.
\end{equation}
By a density argument, we have for any $\varphi\in W_{0}^{1,1}(\cD_{\varepsilon})$,
\begin{equation}\label{weak  v}
\int_{\cD}A^{\alpha\beta}D_{\alpha}vD_{\beta}\varphi+B^{\alpha}D_{\alpha}v\varphi-\hat{B}^{\alpha}v D_{\alpha}\varphi+(\lambda_{0}-C)v\varphi=\int_{\cD}h_{\alpha}D_{\alpha}\varphi.
\end{equation}
Fix $\zeta\in C_{c}^{\infty}(\cD_{\varepsilon})$ with $\zeta\equiv1$ on $\cD'\subset\subset \cD_{\varepsilon}$, and we choose $\varphi=\zeta u\in W_{0}^{1,1}(\cD_{\varepsilon})$ in \eqref{weak  v}. Then
\begin{align}\label{weak  v1}
&\int_{\cD}A^{\alpha\beta}D_{\alpha}v\left(\zeta D_{\beta}u+uD_{\beta}\zeta\right)+B^{\alpha}D_{\alpha}vu\zeta-\hat{B}^{\alpha}v\left(\zeta D_{\alpha}u+uD_{\alpha}\zeta\right)+(\lambda_{0}-C)uv\zeta\nonumber\\
&=\int_{\cD}h_{\alpha}\left(\zeta D_{\alpha}u+uD_{\alpha}\zeta\right).
\end{align}
Recalling that $u\in W^{1,1}(\cD)$ is a weak solution of \eqref{systems}, then by a density argument, for any $\psi\in W_{0}^{1,\infty}(\cD)$, we have
\begin{equation}\label{weak  u}
\int_{\cD}A^{\alpha\beta}D_{\beta}uD_{\alpha}\psi+B^{\alpha}uD_{\alpha}\psi-\hat{B}^{\alpha}D_{\alpha}u\psi+(\lambda_{0}-C)u\psi=\int_{\cD}g_{\alpha}D_{\alpha}\psi-f_{0}\psi.
\end{equation}
By taking $\psi=\zeta v$ in \eqref{weak  u}, we get
\begin{align}\label{weak  u'}
&\int_{\cD}A^{\alpha\beta}D_{\beta}u\left(\zeta D_{\alpha}v+vD_{\alpha}\zeta\right)+B^{\alpha}u\left(\zeta D_{\alpha}v+vD_{\alpha}\zeta\right)-\hat{B}^{\alpha}D_{\alpha}uv\zeta+(\lambda_{0}-C)uv\zeta\nonumber\\
&=\int_{\cD}g_{\alpha}\left(\zeta D_{\alpha}v+vD_{\alpha}\zeta\right)-f_{0}\zeta v.
\end{align}
Combining \eqref{weak  v1} and \eqref{weak  u'}, we find
\begin{align}\label{weak uv}
\int_{\cD}\zeta h_{\alpha}D_{\alpha}u
&=\int_{\cD}-A^{\alpha\beta}vD_{\beta}uD_{\alpha}\zeta+\int_{\cD}A^{\alpha\beta}uD_{\alpha}vD_{\beta}\zeta-uvB^{\alpha}D_{\alpha}\zeta-\hat{B}^{\alpha}uv D_{\alpha}\zeta\nonumber\\
&\quad-\int_{\cD}uh_{\alpha}D_{\alpha}\zeta+\int_{\cD}g_{\alpha}\left(\zeta D_{\alpha}v+vD_{\alpha}\zeta\right)-f_{0}\zeta v.
\end{align}
Thus, by H\"{o}lder's inequality, the Sobolev embedding theorem, $d<p'$, \eqref{W1p' v}, and \eqref{v infty}, we estimate each term on the right-hand side of \eqref{weak uv} by
\begin{align}
&\left|\int_{\cD}A^{\alpha\beta}vD_{\beta}uD_{\alpha}\zeta\right|\leq N\|u\|_{W^{1,1}(\cD_{\varepsilon})}\|v\|_{L^{\infty}(\cD_{\varepsilon})}\leq N\|u\|_{W^{1,1}(D_{\varepsilon})}\|h\|_{L^{p'}(\cD)},\nonumber\\
&\left|\int_{\cD}A^{\alpha\beta}uD_{\alpha}vD_{\beta}\zeta-uvB^{\alpha}D_{\alpha}\zeta-\hat{B}^{\alpha}uv D_{\alpha}\zeta\right|\nonumber\\
&\leq N\|u\|_{L^{\frac{d}{d-1}}(\cD_{\varepsilon})}\|v\|_{W^{1,d}(\cD_{\varepsilon})}
\leq N\|u\|_{W^{1,1}(\cD_{\varepsilon})}\|v\|_{W^{1,p'}(\cD_{\varepsilon})}\leq N\|u\|_{W^{1,1}(\cD_{\varepsilon})}\|h\|_{L^{p'}(\cD)},\nonumber\\
                \label{uh}
&\left|\int_{\cD}uh_{\alpha}D_{\alpha}\zeta\right|\leq N\|u\|_{L^{\frac{d}{d-1}}(\cD_{\varepsilon})}\|h\|_{L^{d}(\cD)}\leq N\|u\|_{W^{1,1}(\cD_{\varepsilon})}\|h\|_{L^{p'}(\cD)},
\end{align}
and
\begin{align*}
&\left|\int_{\cD}g_{\alpha}\left(\zeta D_{\alpha}v+vD_{\alpha}\zeta\right)-f_{0}\zeta v\right|\\
&\leq N\|g\|_{L^{\infty}(\cD)}\|v\|_{W^{1,d}(\cD_{\varepsilon})}+N\|f_{0}\|_{L^{\frac{d}{d-1}}(\cD_{\varepsilon})}\|v\|_{L^{d}(\cD_{\varepsilon})}\\
&\leq N\|g\|_{L^{\infty}(\cD)}\|h\|_{L^{p'}(\cD)}+N\left(\|f\|_{L^{\frac{d}{d-1}}(\cD)}+\|u\|_{L^{\frac{d}{d-1}}(\cD_{\varepsilon})}\right)\|h\|_{L^{p'}(\cD)}\\
&\leq N\left(\|g\|_{L^{\infty}(\cD)}+\|f\|_{L^{\infty}(\cD)}+\|u\|_{W^{1,1}(\cD_{\varepsilon})}\right)\|h\|_{L^{p'}(\cD)}.
\end{align*}
Thus, we get
\begin{align*}
\left|\int_{\cD}\zeta h_{\alpha}D_{\alpha}u\right|\leq N\left(\|g\|_{L^{\infty}(\cD)}+\|f\|_{L^{\infty}(\cD)}+\|u\|_{W^{1,1}(\cD_{\varepsilon})}\right)\|h\|_{L^{p'}(\cD)}.
\end{align*}
It follows from \eqref{uh} that
\begin{align*}
\left|\int_{\cD}h_{\alpha}D_{\alpha}(u\zeta)\right|\leq N\left(\|g\|_{L^{\infty}(\cD)}+\|f\|_{L^{\infty}(\cD)}+\|u\|_{W^{1,1}(\cD_{\varepsilon})}\right)\|h\|_{L^{p'}(\cD)}
\end{align*}
for all $h\in C_{0}^{\infty}(\cD)$. Therefore, $D(u\zeta)\in L^{p}(\cD)$. In particular, $u\in W^{1,p}(\cD')$ for some $1<p<\frac{d}{d-1}$, and
\begin{align*}
\|u\|_{W^{1,p}(\cD')}\leq N\left(\|g\|_{L^{\infty}(\cD)}+\|f\|_{L^{\infty}(\cD)}+\|u\|_{W^{1,1}(\cD_{\varepsilon})}\right).
\end{align*}
Corollary \ref{coro2} is proved.
\end{proof}

\section{Weak type-$(1,1)$ estimates}\label{sec thm3}

In this section, we will prove a global weak type-$(1,1)$ estimate with respect to $A_{1}$ Muckenhoupt weights for solutions to the divergence form systems without lower-order terms, which is the second purpose of this paper. For that, we assume that the sub-domains, $\cD_{1},\dots,\cD_{M-1}$, are away from $\partial\cD$ and impose an additional assumption on the coefficient $A$. Denote $\delta_{0}:=\min_{1\leq j\leq M-1}\{\partial \cD_{j}, \partial\cD\}$. Before continuing to state our second result, we first recall the definition of $A_{1}$ Muckenhoupt weight $w$: We say $w: \mathbb R^{d}\rightarrow[0,\infty)$ belongs to $A_{1}$ if there exists some constant $C$ such that for all balls $B$,
$$\fint_{B}w(y)\ dy\leq C\inf_{x\in B}w(x).$$
The $A_1$ constant $[w]_{A_1}$ of $w$ is defined as the infimum of all such $C$'s.
Moreover, we use the following weighted Sobolev spaces:
$$W_{w}^{1,p}(\cD)=\{u: u, Du\in L_{w}^{p}(\cD)\}.$$
We also use the following notation:
$$
w(\cD)=\int_{\cD}w(x)\ dx\quad \text{and}\quad \|f\|_{L_{w}^{p}(\cD)}:=\int_{\cD}|f|^{p}w\ dx,~ p\in[1,\infty).
$$

\begin{assumption}\label{assump omega}
(1) $A$ is of piecewise Dini mean oscillation in $\cD$, and  there exists some constant $c_{0}>0$ such that for any $r\in(0,1/2)$, $\omega_{A}(r)\leq c_{0}(\ln r)^{-2}$.

(2) For some constant $c_1,c_2>0$, $\omega_0'(R_0^-)\geq c_{1}$ and for any $R\in(0,R_0/2)$, $\omega_0(R)\leq c_{2}(\ln R)^{-2}$.
\end{assumption}

\begin{theorem}\label{thm3}
Let $\cD$ have a $C^{1,\text{Dini}}$ boundary, $p\in(1,\infty)$, $w$ be an $A_{1}$ Muckenhoupt weight, and Assumption \ref{assump omega} be satisfied. For $f\in L_{w}^{p}(\cD)$, let $u\in W_{w}^{1,p}(\cD)$ be a weak solution to
\begin{align*}
\begin{cases}
\mathcal{L}u=\Div f&\quad\mbox{in}~\cD,\\
u=0&\quad\mbox{on}~\partial \cD.
\end{cases}
\end{align*}
Then for any $t>0$, we have
$$w\Big(\{x\in \cD: |Du(x)|>t\}\Big)\leq\frac{N}{t}\|f\|_{L_{w}^{1}(\cD)},$$
where $N$ depends on $n,d,M,\omega_{A},\nu,\Lambda,\varepsilon,\delta_{0},[w]_{A_1}$, the $C^{1,\text{Dini}}$ characteristics of $\cD$ and $\cD_{j}$. Moreover, the linear operator $T: f\mapsto Du$ can be extended to a bounded operator from $L_{w}^{1}(\cD)$ to weak-$L_{w}^{1}(\cD)$.
\end{theorem}

We shall use a generalized version of Lemma \ref{lemma weak} (see the Appendix) since our argument and estimates depend on the coordinate system, as well as the following lemma.
\begin{lemma}[Lemma 3.4 of \cite{dk}]\label{lemma A}
Let $\omega$ be a nonnegative increasing function such that $\omega(t)\leq(\ln\frac{t}{4})^{-2}$ for $0<t\leq 1$, and $\tilde\omega$ be given as in \eqref{tilde phi} with $\omega$ in place of $\bar\omega$. Then for any $r\in(0,1]$, we have
$$\int_{0}^{r}\frac{\tilde\omega(t)}{t}\ dt\leq N\Big(\ln\frac{4}{r}\Big)^{-1},$$
where $N>0$ is some positive constant.
\end{lemma}

Now, we can prove Theorem \ref{thm3}.

\begin{proof}[\bf Proof of Theorem \ref{thm3}.]
To begin with, we note that by using \eqref{relation r R00} in the proof of Lemma \ref{volume} and Assumption \ref{assump omega} (2), we have
$$
\frac{3r}{2}= \int_{0}^{R}s\omega'(s+2r)\ ds\geq\frac{R^{2}}{2}\omega'((R+2r)^-)\geq\frac{c_{1}}{2}R^{2}.$$
Then by Assumption \ref{assump omega} (2), we obtain for any $r\in(0,r_0/2)$,
\begin{equation}\label{eq11.08}
\omega_1(r)=\omega_0(2r+R)\le 2\omega_0(R)\leq 2c_{2}(\ln R)^{-2}\leq c(\ln r)^{-2}
\end{equation}
for some constant $c>0$. We thus conclude that Lemma \ref{lemma A} is available in our case by combining Assumption \ref{assump omega} (1) and \eqref{eq11.08}.

The assumption on $\partial\cD\in C^{1,\text{Dini}}$ implies that \eqref{reifenberg} holds true. Also, the coefficients $A^{\alpha\beta}$ satisfy \eqref{BMO} in the interior of $\cD$ and near the boundary, $A^{\alpha\beta}$ satisfy \eqref{condi boundary}. By Lemma \ref{sol weight}, the $W_{w}^{1,p}$-solvability and estimates for divergence form elliptic systems with $A_{1}$ weights are available. Hence, the map $T: f\mapsto Du$ is a bounded linear map on $L_{w}^{p}(\cD)$. Let $\{Q_{\alpha}^{k}\}$ be a collection of dyadic ``cubes'' as in the proof of \cite[Lemma 4.1]{dek}. By Remark \ref{rmk general}, we can assume that each $Q_{\alpha}^{k}$ is small enough so that they do not intersect with $\cup_{j=1}^{M-1}\overline{\cD_{j}}$ and $\partial\cD$ at the same time. Moreover, for a fixed $x_{k}\in Q_{\alpha}^{k}$, we associate $Q_{\alpha}^{k}$ with a Euclidean ball $B_k=B_{r_k}(x_{k})$ such that $x_{k}\in Q_{\alpha}^{k}\subset B_{k}$, where $r_k=\text{diam}\, Q_{\alpha}^{k}\leq\frac{\delta_{0}}{2}$. Suppose for some $Q_{\alpha}^{k}$ and $t>0$,
\begin{equation}\label{prop f Ql}
t<\frac{1}{w(Q_{\alpha}^{k})}\int_{Q_{\alpha}^{k}}|f|w\ dx\leq C_{2}t.
\end{equation}
Then $f$ admits a decomposition in a given $Q_{\alpha}^{k}$ according to the following dichotomy.

(i) If $\dist(x_{k},\partial\cD)\leq\frac{\delta_{0}}{2}$, then $B_k$ does not intersect with sub-domains $\cD_{j}$, $j=1,\dots,M-1$. In this case, we choose the coordinate system according to $y_{k}\in\partial\cD$, which satisfies $|x_{k}-y_{k}|=\dist(x_{k},\partial\cD)$. Let
$$
g:=\fint_{Q_{\alpha}^{k}}f\ dx,\quad b=f-g\quad\mbox{in}~Q_{\alpha}^{k}.
$$
Then
\begin{equation*}
\fint_{Q_{\alpha}^{k}}b\ dx=0,
\end{equation*}
and
\begin{align*}
|g|\leq \fint_{Q_{\alpha}^{k}}|f|\ dx\leq \frac{1}{|Q_{\alpha}^{k}|\inf\limits_{Q_{\alpha}^{k}}w}\int_{Q_{\alpha}^{k}}|f|w\ dx\leq\frac{1}{w(Q_{\alpha}^{k})}\int_{Q_{\alpha}^{k}}|f|w\ dx\leq C_{2}t,
\end{align*}
where we used the definition of $w$ and \eqref{prop f Ql}.
Hence,
$$\int_{Q_{\alpha}^{k}}|g|^{p}w\ dx\leq C_{2}t^{p}w(Q_{\alpha}^{k}).$$
Let $u_{1}\in W_{w}^{1,p}(\cD)$ be the unique weak solution of
\begin{align*}
\begin{cases}
\mathcal{L}u_{1}=\Div b&\quad\mbox{in}~\cD,\\
u_{1}=0&\quad\mbox{on}~\partial \cD.
\end{cases}
\end{align*}
Set $c=\frac{4R_0}{\delta_{0}}$ with $R_{0}=\mbox{diam}~\cD$. Then for any $R\geq cr_{k}$ such that $\cD\setminus B_{R}(x_{k})\neq\emptyset$ and $h\in C_{0}^{\infty}(\cD_{2R}(x_{k})\setminus B_{R}(x_{k}))$, let $p'=p/(p-1)$, $\mathcal{L}^{*}$ be the adjoint operator of $\mathcal{L}$, and $u_{2}\in W_{w^{-\frac{1}{p-1}}}^{1,p'}(\cD)$ be a weak solution of
\begin{align*}
\begin{cases}
\mathcal{L}^{*}u_{2}=\Div h&\quad\mbox{in}~\cD,\\
u_{2}=0&\quad\mbox{on}~\partial \cD,
\end{cases}
\end{align*}
which satisfies
\begin{equation}
                        \label{eq9.44}
\left(\int_{\cD}|Du_{2}|^{p'}w^{-\frac{1}{p-1}}\ dx\right)^{\frac{1}{p'}}\leq N\left(\int_{\cD}|h|^{p'}w^{-\frac{1}{p-1}}\ dx\right)^{\frac{1}{p'}}=N\left(\int_{\cD_{2R}(x_{k})\setminus B_{R}(x_{k})}|h|^{p'}w^{-\frac{1}{p-1}}\ dx\right)^{\frac{1}{p'}}.
\end{equation}
See Lemma \ref{sol weight}.
Then we can use the definition of adjoint solutions, the fact that $b$ is supported in $Q_{\alpha}^{k}$ with mean zero, and $h\in C_{0}^{\infty}(\cD_{2R}(x_{k})\setminus B_{R}(x_{k}))$ to obtain
\begin{align}\label{esti Du h}
\int_{\cD_{2R}(x_{k})\setminus B_{R}(x_{k})}Du_{1}\cdot h
=\int_{Q_{\alpha}^{k}}Du_{2}\cdot b=\int_{Q_{\alpha}^{k}}\big(Du_{2}-Du_{2}(x_{k})\big)\cdot b.
\end{align}
Since $R\leq R_0$, $B_{\frac{\delta_{0}R}{2R_0}}(x_{k})$ does not intersect with sub-domains $\cD_{j}$, $j=1,\dots,M-1$. Because $\mathcal{L}^{*}u_{2}=0$ in $\cD_{R}(x_{k})$, by flattening the boundary and using a similar argument that led to an a priori estimate of the modulus of continuity of $Du_{2}$ in the proof of Theorem 1.3 in \cite{dek}, we have
\begin{align}\label{case1 est Du}
|Du_{2}(x)-Du_{2}(x_{k})|\leq N\left(\Big(\frac{|x-x_{k}|}{R}\Big)^{\gamma}
+\omega_{A}^{*}(|x-x_{k}|)\right)
R^{-d}\|Du_{2}\|_{L^{1}(\cD_{\frac{\delta_{0}R}{2R_0}}(x_{k}))}
\end{align}
for any $x\in Q_{\alpha}^{k}\subset\cD_{\frac{\delta_{0}R}{4R_0}}(x_{k})$, where $\gamma\in(0,1)$ is a constant and $\omega_{A}^{*}(t)$ is defined as in \cite[(2.34)]{dek}, which is derived from $\omega_{A}(t)$. Then, coming back to \eqref{esti Du h}, using Lemma \ref{lemma A}, \eqref{eq11.08}, the definition of $A_1$ weights, \eqref{case1 est Du}, H\"{o}lder's inequality, and \eqref{eq9.44},
we obtain
\begin{align*}
&\left|\int_{\cD_{2R}(x_{k})\setminus B_{R}(x_{k})}Du_{1}\cdot h\right|\leq\frac{1}{\inf\limits_{Q_{\alpha}^{k}}w}\|b\|_{L_{w}^{1}(Q_{\alpha}^{k})}\|Du_{2}-Du_{2}(x_{k})\|_{L^{\infty}(Q_{\alpha}^{k})}\\
&\leq \frac{NR^{-d}}{\inf\limits_{\cD_{2R}(x_{k})}w}\|b\|_{L_{w}^{1}(Q_{\alpha}^{k})}\|Du_{2}\|_{L^{1}(\cD_{\frac{\delta_{0}R}{2R_0}}(x_{k}))}\left(r_{k}^{\gamma}R^{-\gamma}+\Big(\ln\frac{4}{r_{k}}\Big)^{-1}\right)\\
&\leq\frac{N}{\int_{\cD_{2R}(x_{k})}w\ dx}\|b\|_{L_{w}^{1}(Q_{\alpha}^{k})}\left(\int_{\cD}|Du_{2}|^{p'}w^{-\frac{1}{p-1}}\ dx\right)^{\frac{1}{p'}}\big(\int_{\cD_{2R}(x_{k})}w\ dx\big)^{\frac{1}{p}}\left(r_{k}^{\gamma}R^{-\gamma}+\Big(\ln\frac{4}{r_{k}}\Big)^{-1}\right)\\
&\leq N\big(\int_{\cD_{2R}(x_{k})}w\ dx\big)^{\frac{1}{p}-1}\|b\|_{L_{w}^{1}(Q_{\alpha}^{k})}\left(\int_{\cD_{2R}(x_{k})\setminus B_{R}(x_{k})}|h|^{p'}w^{-\frac{1}{p-1}}\ dx\right)^{\frac{1}{p'}}\left(r_{k}^{\gamma}R^{-\gamma}+\Big(\ln\frac{4}{r_{k}}\Big)^{-1}\right).
\end{align*}
By the duality, we have
$$\|Du_{1}\|_{L_{w}^{p}(\cD_{2R}(x_{k})\setminus B_{R}(x_{k}))}\leq N\big(\int_{\cD_{2R}(x_{k})}w\ dx\big)^{\frac{1}{p}-1}\|b\|_{L_{w}^{1}(Q_{\alpha}^{k})}\left(r_{k}^{\gamma}R^{-\gamma}+\Big(\ln\frac{4}{r_{k}}\Big)^{-1}\right).$$
Therefore, by H\"{o}lder's inequality, we obtain
$$\|Du_{1}\|_{L_{w}^{1}(\cD_{2R}(x_{k})\setminus B_{R}(x_{k}))}\leq N\|b\|_{L_{w}^{1}(Q_{\alpha}^{k})}\left(r_{k}^{\gamma}R^{-\gamma}+\Big(\ln\frac{4}{r_{k}}\Big)^{-1}\right).$$
The rest of proof is similar to that of Lemma \ref{weak est barv}. Hence, we obtain
\begin{align*}
\int_{\cD\setminus B_{cr_{k}}(x_{k})}|Du_{1}|w\ dx\leq N\int_{Q_{\alpha}^{k}}|b|w\ dx\leq N\int_{Q_{\alpha}^{k}}|f|w\ dx+N\int_{Q_{\alpha}^{k}}|g|w\ dx\leq Ntw(Q_{\alpha}^{k}).
\end{align*}
That is,
$$\int_{\cD\setminus B_{cr_{k}}(x_{k})}|Tb\chi_{Q_{\alpha}^{k}}|w\ dx\leq Ntw(Q_{\alpha}^{k}).$$

(ii) If $\dist(x_{k},\partial\cD)\geq\frac{\delta_{0}}{2}$, then $B_k$ does not intersect with $\partial\cD$. In this case, we choose the coordinate system according to $x_{k}$. In a given $Q_{\alpha}^{k}$, we set
\begin{align*}
g(x)=E(x)\fint_{Q_{\alpha}^{k}}E^{-1}(y)f(y)\ dy,\quad b=f-g,
\end{align*}
where $E=(E^{\alpha\beta})$ is a $d\times d$ matrix with
\begin{align*}
E^{\alpha\beta}=\delta_{\alpha\beta}~\, \mbox{for}~1\le \alpha \le d,1\le \beta \le d-1,\quad
E^{\alpha d}=A^{d\alpha}\,~\mbox{for}~1\le \alpha\le d.
\end{align*}
By using the boundedness of $A$, we have
\begin{equation*}
\int_{Q_{\alpha}^{k}}|g|^{p}w\ dx\leq Nt^{p}w(Q_{\alpha}^{k}).
\end{equation*}
Let $\tilde{b}=E^{-1}b$, which has mean zero in $Q_{\alpha}^{k}$. 
We now follow the argument as in (i) and get
\begin{align}\label{iden u1 f11}
\int_{\cD_{2R}(x_{k})\setminus B_{R}(x_{k})}Du_{1}\cdot h&=\int_{Q_{\alpha}^{k}}Du_{2}\cdot b=\int_{Q_{\alpha}^{k}}\left(
          D_{x'}u_{2}, U_{2}
 \right)\cdot\tilde{b}\nonumber\\
&=\int_{Q_{\alpha}^{k}}\left(
          D_{x'}u_{2}-D_{x'}u_{2}(x_{k}), U_{2}-U_{2}(x_{k})
 \right)\cdot\tilde{b},
\end{align}
where $U_{2}=A^{d\beta}D_{\beta}u_{2}$. Recalling that $cr_{k}\leq R\leq R_{0}$, $B_{\frac{\delta_{0}R}{2R_0}}(x_{k})$ does not intersect with $\partial\cD$. By a similar argument that led to \eqref{C1 est1} (or \eqref{C1 est}, \eqref{C1 est2}), for any $x\in Q_{\alpha}^{k}\subset B_{\frac{\delta_{0}R}{4R_0}}(x_{k})$, we have
\begin{align*}
&|(D_{x'}u_{2}(x),U_{2}(x))-(D_{x'}u_{2}(x_{k}),U_{2}(x_{k}))|\nonumber\\
&\leq N\left(\Big(\frac{|x-x_{k}|}{R}\Big)^{\gamma}+\int_{0}^{|x-x_{k}|}\frac{\tilde{\omega}_{A}(t)}{t}\ dt+\omega_1\Big(\frac{|x-x_{k}|}{R}\Big)\right)R^{-d}\|(D_{x'}u_{2},U_{2})\|_{L^{1}(B_{\frac{\delta_{0}R}{2R_0}}(x_{k}))}\nonumber\\
&\leq N\left(\Big(\frac{|x-x_{k}|}{R}\Big)^{\gamma}+\int_{0}^{|x-x_{k}|}\frac{\tilde{\omega}_{A}(t)}{t}\ dt+\omega_1\Big(\frac{|x-x_{k}|}{R}\Big)\right)
R^{-d}\|Du_{2}\|_{L^{1}(B_{\frac{\delta_{0}R}{2R_0}}(x_{k}))}.
\end{align*}
Thus, coming back to \eqref{iden u1 f11} and using the similar argument as in the case (i), we have
\begin{align*}
\int_{\cD\setminus B_{cr_{k}}(x_{k})}|Du_{1}|w\ dx&\leq N\int_{Q_{\alpha}^{k}}|\tilde{b}|w\ dx\leq N\int_{Q_{\alpha}^{k}}|b|w\ dx\\
&\leq N\int_{Q_{\alpha}^{k}}|f|w\ dx+N\int_{Q_{\alpha}^{k}}|g|w\ dx\leq Ntw(Q_{\alpha}^{k}).
\end{align*}
Therefore, $T$ satisfies the hypothesis of Lemma \ref{weak general}, and for any $t>0$,
$$w(\{x\in \cD: |Du(x)|>t\})\leq\frac{N}{t}\|f\|_{L_{w}^{1}(\cD)}.$$
The theorem is proved.
\end{proof}

\section{Appendix}

In the appendix, we give generalizations of Lemmas \ref{lemma weak} and \ref{solvability}. We first recall the definition of the doubling measure $w$: a nontrivial measure on a metric space $X$ is said to be doubling if the measure of any ball is finite and approximately the measure of its double, more precise, if there is a constant $C>0$ such that
$$0<w(B_{2r}(x))\leq Cw(B_{r}(x))<\infty$$
for all $x\in X$ and $r>0$. Let $\cD$ be a bounded domain in $\mathbb R^{d}$ satisfying the condition \eqref{condition D}. We note that $\cD$ equipped with the standard Euclidean metric and the doubling measure $w$ (restricted to $\cD$) is a space of homogeneous type.
By \cite[Theorem~11]{Ch90}, there exists a collection of ``cubes''
$$
\big\{Q^k_\alpha \subset\cD :  k \in \bZ, \; \alpha \in I_k\big\},
$$
with $I_k$ at most countable set and constants $\delta \in (0,1)$, $a_0>0$, and $C_1<\infty$ such that
\begin{enumerate}[i)]
\item
$\cD \setminus \bigcup_\alpha Q^k_\alpha =0\quad \forall~k$.
\item
If $\ell \ge k$ then either $Q^\ell_\beta \subset Q^k_\alpha$ or $Q^\ell_\beta \cap Q^k_\alpha=\emptyset$.
\item
For each $(k,\alpha)$ and each $\ell < k$ there is a unique $\beta$ such that $Q^k_\alpha \subset Q^\ell_\beta$.
\item
$\text{diam}\, Q^k_\alpha \le C_1 \delta^k$.
\item
Each $Q^k_\alpha$ contains some ``ball'' $B_{a_0 \delta^k}(z^k_\alpha) \cap \cD$.
\end{enumerate}
From the above and the doubling property of the measure $w$, we can infer that there is a constant $C_2\ge 1$ such that if $Q^{k-1}_\beta$ is the parent of $Q^k_\alpha$, then we have
\begin{equation*}
w(Q^{k-1}_\beta) \le C_2 w(Q^k_\alpha).
\end{equation*}

Let $p,c\in (1,\infty)$.
\begin{assumption}
                        \label{assump9.26}
i) $T$ is a bounded linear operator on $L_{w}^{p}(\cD)$.

ii) If for some $f\in L_{w}^p(\cD)$, $t>0$, and some cube $Q_\alpha^k$ we have
$$
t<\frac{1}{w(Q_\alpha^k)}\int_{Q_\alpha^k}|f|w\ dx\leq C_{2}t,
$$
then $f$ admits a decomposition $f=g+b$ in $Q_\alpha^k$, where $g$ and $b$ satisfy
\begin{equation}\label{prop h b}
\int_{Q_\alpha^k}|g|^{p}w\ dx\leq C_1t^{p}w(Q_\alpha^k),\quad\int_{\cD\setminus B_{cr}(x_{0})}|T(b\chi_{Q_\alpha^k})|w\ dx\leq C_1tw(Q_\alpha^k)
\end{equation}
with $x_0\in Q_\alpha^k$ and $r=\text{diam}\,Q_\alpha^k$.
\end{assumption}

\begin{remark}\label{rmk general}
By taking a sufficiently large $c$, Assumption \ref{assump9.26} satisfies automatically for large cubes (i.e., small $k$). In fact, when $c\alpha_0\delta^k\ge \text{diam}\,\cD$, we can just take $g=0$ and \eqref{prop h b} holds because  $\cD\setminus B_{cr}(x_{0})=\emptyset$.
\end{remark}

\begin{lemma}\label{weak general}
Under Assumption \ref{assump9.26}, for any $f\in L_{w}^p(\cD)$ and $t>0$, we have
\begin{equation}
                        \label{eq10.58}
w(\{x\in \cD: |Tf(x)|>t\})\leq\frac{N}{t}\int_{\cD}|f|w\ dx,
\end{equation}
where $N=N(d,c,\cD,p,C_1,\|T\|_{L_{w}^p\rightarrow L_{w}^p})$ is a constant. Moreover, $T$ can be extended to a bounded operator from $L_{w}^1(\cD)$ to weak-$L_{w}^1(\cD)$.
\end{lemma}

\begin{proof}
We fix a $k_{0}\in\mathbb{Z}$ and set $\theta=\inf_{\alpha\in I_{k_{0}}}w(Q_{\alpha}^{k_{0}})>0$. See conditions i)--v). Then for any $t>0$, to get \eqref{eq10.58} when
$$\frac{1}{t}\int_{\cD}|f|w\ dx>\theta,$$
it suffices to choose $N\geq\theta^{-1}w(\cD)$. Otherwise,
$$
\frac{1}{w(Q_{\alpha}^{k_{0}})}\int_{Q_{\alpha}^{k_{0}}}|f|w\ dx\leq t,\quad\forall~\alpha\in I_{k_{0}}.
$$
In this case, let $\{Q_{\alpha}^{k}\}$ be a collection of disjoint ``cubes'' from the Calder\'on-Zygmund decomposition as those in the proof of Lemma 4.1 in \cite{dek},
so that we have
\begin{equation*}
t<\frac{1}{w(Q_{\alpha}^{k})}\int_{Q_{\alpha}^{k}}|f|w\ dx\leq C_2t,\quad|f(x)|\leq t~~\mbox{for}~~ a.e.~~x\in\cD\setminus\bigcup_{l}Q_{\alpha}^{k}.
\end{equation*}
We associate each $Q_{\alpha}^{k}$ with a Euclidean ball $B_k=B_{r_k}(x_{k})$, where $r_k=\text{diam}\, Q_{\alpha}^{k}$ and $x_{k}\in Q_{\alpha}^{k}\subset B_{k}$. We denote $B_k^*=B_{cr_k}(x_{k})$.
By Assumption \ref{assump9.26}, in each $Q_{\alpha}^{k}$ we have the decomposition $f=g+b$ and \eqref{prop h b}. We also define $g=f$ and $b=0$ in $\cD\setminus\bigcup_{l}Q_{\alpha}^{k}$. By using the Chebyshev inequality and the assumptions of $f$ and $g$, we have
\begin{align}\label{TH}
&w\Big(\big\{x\in\cD: |Tg(x)|>t/2\big\}\Big)\leq\frac{N}{t^{p}}\int_{\cD}|Tg|^{p}w\ dx\nonumber\\
&\leq\frac{N}{t^{p}}\int_{\cD}|g|^{p}w\ dx\nonumber\\
&\leq\frac{N}{t^{p}}\sum_{l}\int_{Q_{\alpha}^{k}}|g|^{p}w\ dx+\frac{N}{t^{p}}\int_{\cD\setminus\bigcup_{l}Q_{\alpha}^{k}}|g|^{p}w\ dx\nonumber\\
&\leq N\sum_{l}w(Q_{\alpha}^{k})+\frac{N}{t}\int_{\cD\setminus\bigcup_{l}Q_{\alpha}^{k}}|f|w\ dx\leq\frac{N}{t}\int_{\cD}|f|w\ dx.
\end{align}
Next by using \eqref{prop h b}, we obtain
$$
\int_{\cD\setminus B_{k}^{*}}|T(b\chi_{Q_{\alpha}^{k}})|w\ dx\leq Ntw(Q_{\alpha}^{k}),
$$
which implies that
$$
\int_{\cD\setminus \bigcup_l B_{k}^{*}}|Tb|w\ dx
\leq \sum_l \int_{\cD\setminus B_{k}^{*}}|T(b\chi_{Q_{\alpha}^{k}})|w\ dx \le Nt\sum_l w(Q_{\alpha}^{k})\leq N\int_{\cD}|f|w\ dx.
$$
By the Chebyshev inequality, we get
$$w\Big(\{x\in\cD: |Tb(x)|>t/2\}\setminus \bigcup_l B_{k}^{*}\Big)\leq\frac{N}{t}\int_{\cD}|f|w\ dx.$$
Clearly, we also have
$$
w\big(\bigcup_{l}B_{k}^{*}\big)
\le \sum_l w(B_{k}^{*})
\leq N\sum_{l}w(B_{k})\leq N\sum_{l}w(Q_{\alpha}^{k})
\leq\frac{N}{t}\int_{\cD}|f|w\ dx.$$
We thus obtain
$$w\Big(\{x\in\cD: |Tb(x)|>t/2\}\Big)\leq\frac{N}{t}\int_{\cD}|f|w\ dx,$$
which combined with \eqref{TH} finishes the proof of \eqref{eq10.58} since $Tf=Tg+Tb$. The last assertion follows from the fact that $L_{w}^p(\cD)$ is dense in $L_{w}^1(\cD)$. The lemma is proved.
\end{proof}

\begin{remark}
If we take $w=1$ and
$$
g=\fint_{Q_\alpha^k}|f|\ dx,
$$
 Lemma \ref{weak general} is then reduced to Lemma \ref{lemma weak}. In the special case when
$$
g=\frac{1}{W(Q_\alpha^k)} \int_{Q_\alpha^k} f W\ dx,
$$
where $W$ is a nonnegative adjoint solution,
Lemma \ref{weak general} was also essentially used in the proof of \cite[Theorem 1.10]{dek}.
\end{remark}

Now we give the statement of a generalization of Lemma \ref{solvability}. We first give the definition of $A_p$ weights: We say $w: \mathbb R^{d}\rightarrow[0,\infty)$ belongs to $A_{p}$ for $p\in(1,\infty)$ if
$$\sup_{B}\frac{w(B)}{|B|}\left(\frac{w^{\frac{-1}{p-1}}(B)}{|B|}\right)^{p-1}<\infty,$$
where the supremum is taken over all balls in $\mathbb R^{d}$. The value of the supremum is the $A_{p}$ constant of $w$, and will be denoted by $[w]_{A_p}$.

We next consider the domain $\cD$ which is Reifenberg flat and impose the following assumptions on the coefficients $A^{\alpha\beta}$ and the boundary $\partial\cD$, with a parameter $\gamma_0\in (0,1/4)$ to be specified later.
\begin{assumption}[$\gamma_0$] \label{assump coeffi}
There exists a constant $r_{0}\in(0,1]$ such that the following conditions hold.

(1) In the interior of $\cD$, $A^{\alpha\beta}$ satisfy \eqref{BMO} in some coordinate system depending on $x_0$ and $r$.

(2) For any $x_{0}\in\partial\cD$ and $r\in(0,r_{0}]$, there is a coordinate system depending on $x_{0}$ and $r$ such that in this new coordinate system, we have
\begin{equation}\label{reifenberg}
\{(y',y^{d}): x_0^{d}+\gamma_{0}r<y^{d}\}\cap B_{r}(x_0)\subset \cD\cap B_{r}(x_0)\subset\{(y',y^{d}): x_0^{d}-\gamma_{0}r<y^{d}\}\cap B_{r}(x_0),
\end{equation}
and
\begin{equation}\label{condi boundary}
\fint_{B_{r}(x_{0})}|A(x)-(A)_{B'_{r}(x'_{0})}|\,dx\leq\gamma_{0},
\end{equation}
where $(A)_{B'_{r}(x'_{0})}=\fint_{B'_{r}(x'_{0})}A(y',x^{d})\,dy'$.
\end{assumption}

\begin{lemma}\label{sol weight}
Let $p\in(1,\infty)$ and $w$ be an $A_p$ weight. There exists a constant $\gamma_0\in (0,1/4)$ depending on $d$, $p$, $\nu$, $\Lambda$, $\cD$, and $[w]_{A_p}$
such that, under Assumption \ref{assump coeffi}, for any $u\in W_{w}^{1,p}(\cD)$ satisfying
\begin{align}\label{eq weight}
\begin{cases}
D_{\alpha}(A^{\alpha\beta}D_{\beta}u)-\lambda u=\Div f&\quad\mbox{in}~\cD,\\
u=0&\quad\mbox{on}~\partial\cD,
\end{cases}
\end{align}
where $\lambda\geq0$ and $f\in L_{w}^{p}(\cD)$, we have
\begin{equation}\label{W1p es weight}
\|u\|_{W_{w}^{1,p}(\cD)}\leq N\|f\|_{L_{w}^{p}(\cD)},
\end{equation}
where $N=N(n,d,p,\nu,\Lambda,[w]_{A_{p}},r_{0})$. Furthermore, for any $f\in L_{w}^{p}(\cD)$, \eqref{eq weight} admits a unique solution $u\in  W_{w}^{1,p}(\cD)$.
\end{lemma}

\begin{proof}
For $\lambda>\lambda_{0}$, where $\lambda_{0}>0$ is a sufficiently large constant depending on $n,d,p,\nu,\Lambda,[w]_{A_{p}}$ and $r_{0}$, then the solvability for the operator $D_{\alpha}(A^{\alpha\beta}D_{\beta})-\lambda I$ is proved in \cite[Section 8]{dk1}, and \cite[Theorem 7.2]{dk1} implies that \eqref{W1p es weight} holds true. For $0\leq\lambda\leq\lambda_{0}$, We rewrite \eqref{eq weight} as
\begin{align}\label{eq weight1}
\begin{cases}
D_{\alpha}(A^{\alpha\beta}D_{\beta}u)-(\lambda+\lambda_{0})u=\Div f-\lambda_{0}u&\quad\mbox{in}~\cD,\\
u=0&\quad\mbox{on}~\partial\cD.
\end{cases}
\end{align}
Then by \cite[Theorem 7.2]{dk1}, we have
\begin{equation}\label{appen es u W1p}
\|u\|_{W_{w}^{1,p}(\cD)}\leq N\big(\|u\|_{L_{w}^{p}(\cD)}+\|f\|_{L_{w}^{p}(\cD)}\big),
\end{equation}
where $N=N(n,d,p,\nu,\Lambda,[w]_{A_{p}},r_{0})$. Thus by the method of continuity, it suffices to bound $\|u\|_{L_{w}^{p}(\cD)}$ by $\|f\|_{L_{w}^{p}(\cD)}$. By using H\"{o}lder's inequality, $f\in L_{w}^{p}(\cD)$, and the self-improving property of $A_{p}$ weights, we have for some small $\delta_{1}>0$ depending on $d$, $p$, and $[w]_{A_p}$,
\begin{align}\label{f L1+delta_{1}}
\int_{\cD}|f|^{1+\delta_{1}}\ dx&\leq \left(\int_{\cD}|f|^{p}w\ dx\right)^{\frac{1+\delta_{1}}{p}}\left(\int_{\cD}w^{-\frac{1+\delta_{1}}{p-1-\delta_{1}}}\ dx\right)^{1-\frac{1+\delta_{1}}{p}}\nonumber\\
&\leq\left(\int_{\cD}|f|^{p}w\ dx\right)^{\frac{1+\delta_{1}}{p}}\left(\int_{\cD}w^{-\frac{1}{p-1}}\ dx\right)^{\frac{p-1}{p}(1+\delta_{1})}<\infty.
\end{align}
Next, in view of the reverse H\"{o}lder's inequality for $A_{p}$ weights, we have $w\in L^{1+\delta_{2}}$ for some $\delta_{2}>0$ depending on $d$, $p$, and $[w]_{A_p}$. By using H\"older's inequality, Young's inequality, the weighted Sobolev embedding theorems (see \cite[Theorem 1.3]{ecr}), and \eqref{appen es u W1p}, we have
\begin{align*}
\|u\|_{L_{w}^{p}(\cD)}&\leq \|u\|_{L_{w}^{p_{1}}(\cD)}^{\theta}\|u\|_{L_{w}^{\delta_{3}}(\cD)}^{1-\theta}\\
&\leq \varepsilon\|u\|_{L_{w}^{p_{1}}(\cD)}+N\|u\|_{L_{w}^{\delta_{3}}(\cD)}\\
&\leq \varepsilon\|u\|_{W_{w}^{1,p}(\cD)}+N\|u\|_{L_{w}^{\delta_{3}}(\cD)}\\
&\leq N'\varepsilon\|u\|_{L_{w}^{p}(\cD)}+N\|f\|_{L_{w}^{p}(\cD)}+N\|u\|_{L_{w}^{\delta_{3}}(\cD)},
\end{align*}
which yields
\begin{equation}\label{weight u W1p}
\|u\|_{L_{w}^{p}(\cD)}\leq N\big(\|f\|_{L_{w}^{p}(\cD)}+\|u\|_{L_{w}^{\delta_{3}}(\cD)}\big)
\end{equation}
if we choose $\varepsilon$ sufficiently small so that $N'\varepsilon<1/2$, where
$$
p_{1}=\frac{dp}{d-1}>p,\quad \frac{\theta}{p_{1}}+\frac{1-\theta}{\delta_{3}}=\frac{1}{p},
$$
and $\delta_{3}\in (0,1/2)$ is to be determined below. It follows from H\"{o}lder's inequality and $w\in L^{1+\delta_{2}}$ that
\begin{align*}
\int_{\cD}|u|^{\delta_{3}}w\ dx&\leq\left(\int_{\cD}|u|^{1+\delta_{1}}\ dx\right)^{\frac{\delta_{3}}{1+\delta_{1}}}\left(\int_{\cD}w^{\frac{1+\delta_{1}}{1+\delta_{1}-\delta_{3}}}\ dx\right)^{1-\frac{\delta_{3}}{1+\delta_{1}}}\\
&=\left(\int_{\cD}|u|^{1+\delta_{1}}\ dx\right)^{\frac{\delta_{3}}{1+\delta_{1}}}\left(\int_{\cD}w^{1+\delta_{2}}\ dx\right)^{\frac{1}{1+\delta_{2}}},
\end{align*}
where we chose $\delta_{3}=\frac{\delta_{2}(1+\delta_{1})}{1+\delta_{2}}$. Therefore,
\begin{equation}\label{appen es u delta 3}
\|u\|_{L_{w}^{\delta_{3}}(\cD)}\leq N\|u\|_{L^{1+\delta_{1}}(\cD)}.
\end{equation}
Coming back to \eqref{weight u W1p}, using \eqref{appen es u delta 3}, \cite[Theorem 8.6 (iii)]{dk2}, and \eqref{f L1+delta_{1}}, we have
\begin{equation*}
\|u\|_{L_{w}^{p}(\cD)}\leq N\big(\|f\|_{L_{w}^{p}(\cD)}+\|u\|_{L^{1+\delta_{1}}(\cD)}\big)\leq N\big(\|f\|_{L_{w}^{p}(\cD)}+\|f\|_{L^{1+\delta_{1}}(\cD)}\big)\leq N\|f\|_{L_{w}^{p}(\cD)}.
\end{equation*}
Then combining it with \eqref{appen es u W1p}, we get
$$\|u\|_{W_{w}^{1,p}(\cD)}\leq N\|f\|_{L_{w}^{p}(\cD)}.$$
Therefore, \eqref{W1p es weight} is proved.
\end{proof}


\noindent{\bf{\large Acknowledgements.}} This work was completed while the second author was visiting Brown University. She would like to thank the Division of Applied Mathematics at Brown University for the hospitality and the stimulating environment.

\end{document}